\documentclass{amsart}
\usepackage[utf8]{inputenc}

\usepackage[all,cmtip]{xy}

\usepackage{amsmath}
\usepackage{amssymb}
\usepackage{hyperref}
\usepackage{epsfig}
\usepackage{graphicx}
\numberwithin{equation}{section}

\input xy
\xyoption{all}

\calclayout
\allowdisplaybreaks[3]

\DeclareFontFamily{U}{mathb}{\hyphenchar\font45}
\DeclareFontShape{U}{mathb}{m}{n}{
      <5> <6> <7> <8> <9> <10> gen * mathb
      <10.95> mathb10 <12> <14.4> <17.28> <20.74> <24.88> mathb12
      }{}
\DeclareSymbolFont{mathb}{U}{mathb}{m}{n}
\DeclareMathSymbol{\righttoleftarrow}{3}{mathb}{"FD}

\theoremstyle{plain}
\newtheorem{prop}{Proposition}

\newtheorem{theo}[prop]{Theorem}
\newtheorem{coro}[prop]{Corollary}

\theoremstyle{definition}

\newtheorem{rema}[prop]{Remark}

\newtheorem{exam}[prop]{Example}

\newtheorem*{cond*}{Condition (H1)}

\def\ra{\rightarrow}

\newcommand{\eqto}{\stackrel{\lower1.5pt\hbox{$\scriptstyle\sim\,$}}\to}
\newcommand{\eqdashto}{\stackrel{\lower1.5pt\hbox{$\scriptstyle\sim\,$}}\dashrightarrow}

\def\cA{{\mathcal A}}

\def\cE{{\mathcal E}}

\def\cG{{\mathcal G}}

\def\cO{{\mathcal O}}
\def\cP{{\mathcal P}}
\def\cQ{{\mathcal Q}}

\def\cX{{\mathcal X}}

\def\fS{{\mathfrak S}}

\def\mD{{\mathfrak D}}

\def\bA{{\mathbb A}}
\def\bG{{\mathbb G}}
\def\bP{{\mathbb P}}
\def\bQ{{\mathbb Q}}

\def\bZ{{\mathbb Z}}

\def\bN{{\mathbb N}}
\def\bC{{\mathbb C}}

\def\rH{{\mathrm H}}

\def\Alb{\mathrm{Alb}}
\def\Aut{\mathrm{Aut}}

\def\Br{\mathrm{Br}}

\def\Bl{\mathrm{Bl}}

\def\Pic{\mathrm{Pic}}
\def\End{\mathrm{End}}

\def\Gr{\mathrm{Gr}}
\def\Kum{\mathrm{Kum}}

\def\SL{\mathrm{SL}}

\def\GL{\mathrm{GL}}
\def\PGL{\mathrm{PGL}}

\def\Hom{\mathrm{Hom}}

\def\lim{\mathrm{lim}}

\def\Spec{\mathrm{Spec}}
\def\Sym{\mathrm{Sym}}

\newcommand{\NS}{\operatorname{NS}}

\makeatother
\makeatletter

\author{Brendan Hassett}
\address{Department of Mathematics\\
Brown University \\
Box 1917 
151 Thayer Street
Providence, RI 02912 \\
USA}
\email{brendan\underline{ }hassett@brown.edu}

\author{Yuri Tschinkel}
\address{Courant Institute\\
                New York University \\
                New York, NY 10012 \\
                USA }
\address{Simons Foundation\\
160 Fifth Avenue\\
New York, NY 10010\\
USA}

\email{tschinkel@cims.nyu.edu}

\title[Complete intersections of two quadrics]{Equivariant geometry of odd-dimensional complete intersections of two quadrics}
\begin{document}
\date{\today}

\dedicatory{To Herb Clemens, with admiration}

\maketitle

\begin{abstract}
Fix a finite group $G$. We seek to classify varieties with $G$-action equivariantly birational to a 
representation of $G$ on affine or projective space. Our focus is odd-dimensional smooth complete intersections
of two quadrics, relating the equivariant rationality problem with analogous Diophantine questions over nonclosed fields. We explore
how invariants -- both classical cohomological invariants and recent symbol constructions -- control rationality in some cases.
\end{abstract}

\section{Introduction}
\label{sect:int}

Let $X\subset \bP^{2g+1}$ be a smooth complete intersection of two quadrics over an algebraically closed field $k$
of characteristic zero. We are particularly interested in these varieties because they have a rich birational structure, which can be completely understood in small dimensions. They also have beautiful connections to hyperelliptic curves and are key examples in the theory of 
of intermediate Jacobians.  

In this paper, we study these varieties from the perspective of their equivariant geometry, 
for regular generically free actions of finite groups. 
The main problem is to distinguish such actions up to equivariant birational equivalence, and in particular, to determine which of these are {\em linearizable}, i.e., equivariantly birational to a linear action on $\bP^{2g-1}$. This shares
many similarities with the study of these varieties over nonclosed fields, but has important special features.  

To address this problem, we examine various canonical constructions:
\begin{itemize}
\item{automorphism groups and their induced actions on geometric invariants;}
\item{the structure of varieties of linear subspaces on $X$ and associated pencils of quadric hypersurfaces;}
\item{intermediate Jacobians and their principal homogeneous spaces.}
\end{itemize} 
We elaborate on constructions of Reid \cite{ReidThesis}, Desale-Ramanan \cite{DR}, Donagi \cite{Donagi}, 
and Bhargava-Gross-Wang \cite{BGW,wang} from a functorial/moduli perspective applicable to equivariant
geometry. We also present a new connection with hyperk\"ahler geometry (see Section~\ref{sect:lagrangian}),
extending Kummer-type constructions to higher dimensions; connections between Fano and hyperk\"ahler geometry are in
the focus of many recent studies, including \cite{FMOS}.  

Recent work \cite{BenWit1,HT19,HT1,BenWit2,KP1} addresses rationality questions for geometrically-rational threefolds
over nonclosed fields. Our principal theorem (Theorem~\ref{thm:main-qua}) demonstrates how these results translate
into equivariant contexts: for smooth complete intersections of two quadrics in $\bP^5$, rationality is governed by the existence
of lines.

In Section~\ref{sect:AI}, we present fundamental notions of $G$-equivariant rationality and related cohomological 
invariants. We summarize key geometric structures arising from odd-dimensional complete intersections of two quadrics
in Section~\ref{sect:geometry}.  Equivariant constructions and results over nonclosed fields are developed in parallel.
The resulting principal homogeneous spaces and their embeddings are explored in
Section~\ref{sect:PHS}. Section~\ref{sect:lagrangian} breaks from the main narrative to make a connection with hyperk\"ahler
manifolds in small dimensions. The rest of the paper focuses on rationality problems. Two key constructions are presented
in Section~\ref{sect:RatCon}. Section~\ref{sect:NCfield} relates existence of fixed points to the analogous questions on 
rational points over function fields. 
We close with detailed analysis of the three-dimensional case in Section~\ref{sect:ADT}, highlighting both generic behavior and the special properties
of examples with large automorphism groups. This brings into sharp relief
the similarities and differences between equivariant geometry and geometry
over nonclosed fields.

\

\noindent
{\bf Acknowledgments:} The first author was partially supported by Simons Foundation Award 546235 and 
NSF grant 1701659, the second author by NSF grant 2000099.
We thank Aaron Landesman and Daniel Litt for pointing out a gap in a previous version of this paper, as well as the referees for their
extraordinary care reviewing our manuscript.

\section{Actions and invariants}
\label{sect:AI}
In this section, the base field $k$ is algebraically closed of characteristic zero.

Let $G$ be a finite group and $X$ a smooth projective (connected) variety with a regular $G$-action; we call such varieties $G$-varieties. We say that 
$G$-varieties $X,Y$ are $G$-birational if there exists a $G$-equivariant birational map $X\dashrightarrow Y$; 
stable $G$-birationality means $G$-birationality of 
$X\times \bP^n$ and $Y\times \bP^m$, with trivial $G$-actions on the second factors. Of particular interest are cases 
when $Y=\bP^n$ is projective space, with {\em linear} $G$-action, which we describe below.

\subsection{Linear actions}
\label{subsect:linearaction}

An action of a finite group $G$ on $\bP^n$ is given by a representation
$G \ra \PGL(V)$ where $V=\bA^{n+1}$.
A linear action on $\bP^n$ may have at least two different meanings:
\begin{itemize}
\item{{\em strictly linear action}: the projectivization of a linear representation $G \ra \GL(W \oplus 1)$ where
$W=\bA^n$ and $1$ is the trivial representation;}
\item{{\em linear action:} the projectivization of a linear representation $G \ra \GL(V)$ where $V=\bA^{n+1}$;}
\end{itemize}
We record a few obvious facts:
\begin{itemize}
\item{strictly linear actions admit fixed points;}
\item{if $L$ is a one-dimensional representation of $G$ then the representations
$V$ and $V\otimes L$ give rise to the same projective actions;}
\item{using the exact sequence
$$1 \rightarrow \mu_{n+1} \rightarrow \SL_{n+1} \rightarrow \PGL_{n+1} \rightarrow 1
$$
any projective action lifts to a linear action for a central extension
$$1 \rightarrow \mu_{n+1} \rightarrow \widetilde{G} \rightarrow G \rightarrow 1
$$
and the exact sequence
$$1 \rightarrow \bG_m \rightarrow \GL_{n+1} \rightarrow \PGL_{n+1} \rightarrow 1 $$
lifts any projectively linear action to a representation on a central extension
$$ 1 \rightarrow \bG_m \rightarrow \widehat{G} \rightarrow G \rightarrow 1;$$
}
\item{given a projective representation $\rho:G \ra \PGL_{n+1}$, the resulting cohomology class
\begin{equation} \label{alpha}
\alpha(\rho) \in \rH^2(G,k^{\times}), \quad (n+1)\alpha(\rho)=0,
\end{equation}
measures the obstruction to lifting $\rho$ to a linear representation of $G$.}
\end{itemize}
One last observation: Suppose we are given a projective action of $G$ with a fixed point
$p \in \bP^n$. We obtain a lift 
$$\widetilde{G} \rightarrow \SL(V)$$
and the preimage of $p$ gives a one-dimensional subspace $L\subset V$. The tensor product $V\otimes L^{-1}$
is also a representation of $\widetilde{G}$ on which $\mu_{n+1}$ acts trivially, thus descends to a 
representation of $G$. Thus we find:
\begin{prop} \label{prop:fixstrict}
A projective representation with fixed point is strictly linear and the class $\alpha$ vanishes.  
\end{prop}

If there is an equivariant embedding
$$X \hookrightarrow \bP^N$$
such that $\bP^N$ admits no fixed points then the same holds true for $X$.  

\subsection{Notions of rationality}
We say that the $G$-action on $X$ is {\em projectively linear} if $X$ admits a 
$G$-equivariant birational map to $\bP^n,n=\dim(X)$; it is {\em strictly linear} or {\em linear} if the 
$G$-action on $\bP^n$ has the same properties.  (We think of these as equivariant analogs of 
rationality over nonclosed fields.) When $G$ is abelian, the existence of a fixed point is a birational invariant of smooth projective $G$-varieties \cite[Appendix]{RYKS}. Thus for abelian actions, the existence of a fixed point is a necessary condition for strict linearity.  

The classification of rational $G$-varieties has been essentially settled in dimension two \cite{DI},
but is largely open in higher dimensions. The {\em birational} classification of finite group actions on projective space
remains a challenging and interesting problem \cite{KTprojective}.

Two strictly linear and generically-free actions of $G$ on projective space need not be $G$-birational
but we shall see that they are necessarily stably $G$-birational. The key ingredient is:
\begin{prop}
Suppose $X$ and $Y$ are smooth varieties
with generically-free $G$-actions. If there exist $G$-equivariant vector bundles $E\ra X$ and 
$F\ra Y$ such that $E$ and $F$ are equivariantly birational then $X$ and $Y$ are equivariantly stably
birational. 
\end{prop}
This is a corollary of the `No-name Lemma' \cite[\S 4.3]{ChGR}: $E$ is $G$-equivariantly
birational to $\bA^{\operatorname{rank}(E)} \times X$ over $X$, where the affine factor has trivial $G$-action.
One can find examples of $X$ and $Y$ that are not $G$-birational using the invariants of  \cite{ReichsteinYoussin}.

As a corollary, the following notions of $G$-equivariant stable birational equivalence coincide:
\begin{itemize}
\item{$X \times \bA^m$ and $Y\times \bA^n$ are $G$-equivariantly birational, where the affine spaces
have trivial $G$-actions;}
\item{$X \times V$ and $Y\times W$ are $G$-equivariantly birational, where $V$ and $W$ are linear
representations of $G$;}
\item{$X\times \bP^m$ and $Y \times \bP^n$ are $G$-equivariantly birational, where the actions on the projective spaces
are strictly linear or admit a fixed point.}
\end{itemize}
The last statement follows from Proposition~\ref{prop:fixstrict}.

\subsection{Picard and Brauer groups}
\label{subsect:PBG}
We continue to assume that $G$ is a finite group acting regularly on a smooth projective variety $X$. 
We refer the reader to \cite{BrionLinearization}, \cite{KKV} and \cite{KKLV} for background on line bundles and group actions.

A {\em $G$-linearized line bundle} $L\rightarrow X$ consists of a line bundle $L$ over $X$ and an action of $G$ on $L$, compatible with the action on $X$,
such that the induced action on the fibers is linear.  It follows that $\Gamma(X,L^{\otimes N}), N\in \bZ,$ is naturally a representation of $G$. Conversely, if $X\hookrightarrow \bP^n$ is $G$-equivariant with $G$ acting linearly on $\bP^n$ then $L=\cO_{\bP^n}(-1)|X$ -- the restriction of the universal line on $\bP^n$
to $X$ -- has a natural linearization. The {\em equivariant Picard group} $\Pic_G(X)$ parametrizes $G$-linearized line bundles on $X$, up to equivariant isomorphism. We have an exact sequence cf.~\cite[Th.~1]{FW1}
\begin{equation}0 \rightarrow G^\vee \rightarrow \Pic_G(X) \rightarrow \Pic(X)^G \rightarrow \rH^2(G,k^\times) \label{ES1}
\end{equation}
where $G^{\vee}=\Hom(G,k^{\times})$.
Given a $G$-equivariant $X \hookrightarrow \bP^n$ with $G$ acting projectively on $\bP^n$, the final coboundary morphism 
applied to $\cO_{\bP^n}(-1)|X$ vanishes precisely when the class $\alpha=0$ (see (\ref{alpha}).
In particular, for $L\in \Pic(X)^G$ some power $L^{\otimes N},N\neq 0,$ admits a linearization because $\rH^2(G,k^{\times})$ is a torsion group.

Let $\Br_G(X)$ denote the equivariant Brauer group of $X$, following Fr\"ohlich and Wall \cite{FW1,FW2}.  It parametrizes Azumaya algebras
$\cA \rightarrow X$ with $G$-action, compatible with the action on $X$ and linear over the fibers.  (Azumaya algebras over commutative rings are
direct generalizations of central simple algebras over fields.) We mod out by those of the form 
$\End(\cE)$, where $\cE$ is a $G$-equivariant locally-free sheaf on $X$. It may be computed with a Hochschild-Serre type spectral sequence
cf.~\cite[\S 4]{FW2}, with graded pieces
\begin{align*}
\operatorname{coker} \bigl( \Pic(X)^G &\rightarrow \rH^2(G,k^\times)\bigr) \\
\ker \bigl(\rH^1(G,\Pic(X)) &\rightarrow \rH^3(G,k^\times) \bigr) \\  
\ker \Bigl( \ker \bigl( \Br(X)^G \rightarrow \rH^2(G,\Pic(X)) \bigr)  & \rightarrow  \operatorname{coker} \bigl(\rH^1(G,\Pic(X)) \rightarrow \rH^3(G,k^\times) \bigr) \Bigr).
\end{align*}
Part of this is summarized in the extension to (\ref{ES1}) \cite[Th.~1]{FW1}:
\begin{align*}
0 \rightarrow& G^\vee \rightarrow \Pic_G(X) \rightarrow \Pic(X)^G \rightarrow \rH^2(G,k^\times) \ra \\
                    & \ker(\Br_G(X) \rightarrow \Br(X)) \rightarrow \rH^1(G,\Pic(X)) \rightarrow \rH^3(G,k^{\times}).
\end{align*}
\begin{prop} \label{prop:projectiveBrauer}
Let $V$ be a finite-dimensional linear representation of $G$ over $k$. Then the homomorphism
\begin{equation} \label{eqn:pull}
\rH^2(G,k^{\times}) = \Br_G(\mathrm{point}) \rightarrow \Br_G(\bP(V)),
\end{equation}
induced by the structure morphism, is an isomorphism.

Suppose that $\rho:G \rightarrow \PGL(V)$ is a projective representation. 
Then the kernel of (\ref{eqn:pull}) contains $\alpha(\rho)$.
\end{prop}
\begin{coro} \label{coro:product}
Let $X$ be a smooth projective $G$-variety and $E\ra X$ a $G$-equivariant vector bundle.
Then the induced homomorphism
$$\Br_G(X) \rightarrow \Br_G(\bP(E))$$
is an isomorphism.
\end{coro}
\begin{proof}[Proof of Proposition~\ref{prop:projectiveBrauer} and Corollary~\ref{coro:product}]
The first statement of the proposition follows from applying the Hochschild-Serre formalism to $\bP(V)$.
We are applying $\Pic(\bP(V))=\bZ$ (with trivial $G$ action), $\Br(\bP(V))=0$, and exact sequence (\ref{ES1}).  
For the second, note that projection
$$\bP(V) \times \bP(V) \rightarrow \bP(V)$$
admits the diagonal section so Proposition~\ref{prop:fixstrict} gives the vanishing of $\alpha(\rho)$ on
pullback to $\bP(V)$.  The corollary follows by computing the \'etale Leray spectral sequence for $\bP(E)\rightarrow X$,
using the vanishing underlying Proposition~\ref{prop:projectiveBrauer}.
\end{proof}
\begin{prop}$\Br_G$ is an equivariant stable birational invariant of smooth projective $G$-varieties. 
\end{prop}
The stable birational invariance of the Brauer group is well-known \cite[\S 5.2]{CTSko}.  
\begin{proof}
We first prove birational invariance.
By weak factorization, it suffices to prove the assertion for a blow up
$$
\Bl_Z(X) \rightarrow X,
$$
where $Z\subset X$ is smooth and irreducible as a $G$-variety, i.e., $G$ permutes the connected components transitively. 
In this situation, we observe that
\begin{itemize}
\item{The exceptional divisor $E$ is $G$-invariant and irreducible, thus the bottom graded piece of $\Br_G$ is unchanged.}
\item{The connected components of $E$ generate a permutation module for $G$, with trivial $\rH^1$, thus the 
middle graded piece is unchanged.}
\item{We have $\Br(\Bl_Z(X))=\Br(X)$ thus the top graded piece is unchanged.}
\end{itemize}

For the stable case, we need to verify that
$$\Br_G(X \times \bP^n) = \Br_G(X)$$
provided the $G$-action on $\bP^n$ is linear.  We compute this using the Leray spectral sequence associated
with the projection $X\times \bP^n \rightarrow X$. The vanishing used in the proof of Proposition~\ref{prop:projectiveBrauer}
gives the desired equality.  
\end{proof}

\begin{coro}
Let $X$ be a smooth projective $G$-variety.  Assume there is a $G$-equivariant embedding
$$X \hookrightarrow \bP^n$$
where the action of $G$ on $\bP^n$ is not linear. Then $X$ is not equivariantly birational to projective space
with a linear $G$-action.
\end{coro}
Indeed, we can factor
$$\Br_G({\mathrm{point}}) \rightarrow \Br_G(\bP^n) \rightarrow \Br_G(X)$$
and it suffices to exhibit nonzero elements in the kernel of the first homomorphism.  
However, the class $[\cO_{\bP^n}(1)] \in \Pic^G(\bP^n)$ 
maps to a nontrivial $\alpha \in \rH^2(G,k^{\times})$. The spectral sequence above shows that $\alpha$ is in 
the kernel of $\Br_G(\mathrm{point}) \rightarrow \Br_G(\bP^n)$.

\begin{rema}
The paper \cite[\S 6]{BCDP} introduces similar ideas via the {\em Amitsur subgroup}, defined as the image
of $\Pic(X)^G \rightarrow \rH^2(G,k^\times)$.

Much of this extends to a nonclosed field $k$, except for the interpretation of
$\Br_G(\Spec(k))$ as the group cohomology for $G$. Moreover, we 
would have to keep track of the Galois actions on $\mu_n$ and $\widetilde{G}$.
\end{rema}

\section{Algebraic geometry of pencils of quadrics}
\label{sect:geometry}

Here we assume that the ground field is algebraically closed of characteristic zero and all objects
are equivariant for the action of a finite group $G$. The discussion below is valid, with minor changes,
when the objects are defined over a field of characteristic zero. 

We start with a projective representation $\rho:G \ra \PGL_{2g+2}$ corresponding to a 
$G$-action on $\bP^{2g+1}$.  Let $\rho^*$ be the dual representation,
$\Sym^2(\rho^*)$ the symmetric square, and $\wedge^2(\Sym^2(\rho^*))$ its second
exterior power;
note that tensor powers of projective representations are well-defined as projective
representations. 

Consider a smooth complete intersection of two quadrics $X\subset \bP^{2g+1}$.
We may write 
$$
X=\{Q_1=Q_2=0\},
$$
where $Q_1$ and $Q_2$ are basis elements for the distinguished two-dimensional
subrepresentation of $\Sym^2(\rho^*)$ generating the ideal of $X$.
(The elements $Q_1$ and $Q_2$ need not be invariant under the action of $G$.) 
Now $\wedge^2(\Sym^2(\rho))$ has a fixed point -- the pencil -- so Proposition~\ref{prop:fixstrict} gives
$$\alpha(\wedge^2(\Sym^2 (\rho)))=4\alpha(\rho)=0.$$

We write
$$\cQ=\Bl_X(\bP^{2g+1})=\{t_1Q_1+t_2Q_2=0\} \subset \bP^{2g+1} \times \bP^1$$
for the corresponding pencil and 
$$q:\cQ \ra \bP^1$$
for the projection onto the second factor, a fibration in quadric hypersurfaces.

We recall fundamental results from \cite{ReidThesis} and \cite{Donagi}.
The singular members of the pencil $\cQ$ are given by the degeneracy locus
$$B=\{ \det(t_1Q_1+t_2Q_2)=0\} \subset \bP^1,$$
which consists of $2g+2$ distinct points $b_1,\ldots,b_{2g+2}$, each corresponding to a nodal 
fiber $\cQ_{b_i}$.  The action of $G$ induces a permutation of the $b_i$.  

Let $F_g(q) \ra \bP^1$ denote the relative variety of maximal isotropic subspaces 
of $q:\cQ \rightarrow \bP^1$, with Stein factorization
$$F_g(q) \ra C \ra \bP^1,$$
where the first arrow is smooth.  The double cover $C\ra \bP^1$ 
encodes the two connected components of the variety of maximal isotropic subspaces.  
The pullback of the point class on $\bP^1$ to $C$ is written $g^1_2$, an element of
$(\Pic^2(C))^G$.  
We have natural bijections between
\begin{itemize}
\item{the branch points $b_1,\ldots,b_{2g+2}$ of $C\rightarrow \bP^1$;}
\item{the nodes of members of the pencil $\cQ_t$.}
\end{itemize}
The points $b_i$ -- regarded as ramification points on $C$ -- generate a subgroup of $\Pic(C)$ presented as follows:
\begin{itemize}
\item{$2b_i=2b_j=g^1_2$ for all $i,j$;}
\item{$b_1+\cdots+b_{2g+2}=(g+1)g^1_2$.}
\end{itemize}
The elements $b_i-b_j$ generate $J(C)[2]$, the two-torsion of the Jacobian of $C$.

The relative variety of maximal isotropic subspaces $F_g(q)$ is isomorphic to the variety parametrizing
$(g-1)$-dimensional quadric hypersurfaces contained in $X$, which is stratified by rank
$$K_0(X) \subset K_1(X) \cdots \subset K_g(X)=F_g(q).$$

Consider the variety $F_{g-1}(X) \subset \Gr(g,2g+2)$ parametrizing $(g-1)$-dimensional linear subspaces contained in $X$. We have:
\begin{itemize} 
\item{$F_{g-1}(X)$ is a principal homogeneous space over the Jacobian $J(C)$.}
\item{There is a correspondence
$$
\begin{array}{ccc}
\widetilde{K_1(X)} & \ra & K_1(X) \\
  \downarrow &  & \\
F_{g-1}(X) \times F_{g-1}(X)  & & 
     \end{array}
$$
where the horizontal arrow is the double cover reflecting the support of elements of $K_1(X)$.}
\item{The correspondence induces a morphism
\begin{equation} \label{morphism:Wang}
\Sym^2(F_{g-1}(X)) \rightarrow \Pic^1(C)
\end{equation}
taking $K_1(X)$ to $C \subset \Pic^1(C)$, with fibers
Kummer varieties singular along the $2^{2g}$ points of $K_0(X)$ over each
point of $C$. Work of X.~Wang \cite{wang} establishes the formula
\begin{equation}
2[F_{g-1}(X)]=[\Pic^1(C)], \label{eqn:Wang}
\end{equation}
as principal homogenous spaces over the Jacobian $J(C)$ of $C$,
i.e., as elements of the Weil-Ch{\^a}telet group of $J(C)$.}
\item{$K_0(X)$ may be interpreted as a 
$J(C)[2]$-principal homogeneous space over $C$.
We have $2^{2g}(2g+2)$ distinguished points of $K_0(X)$ corresponding to the elements lying 
over the Weierstrass points of $C$.}
\item{The automorphisms of $X$ act faithfully on $K_0(X)$ but automorphisms lifting the hyperelliptic 
involution of $C$ fix the $2^{2g}(2g+2)$ distinguished points.}
\end{itemize}

We summarize the implications of the discussion above for the automorphisms:
\begin{prop}
We have an extension
$$1 \rightarrow J(C)[2] \rightarrow \Aut(X) \rightarrow \Aut(X,C) \rightarrow 1,
$$
where $\Aut(X,C)$ is the image of $\Aut(X) \rightarrow \Aut(C)$. Moreover, 
\begin{itemize}
\item $\Aut(X,C)$ contains the hyperelliptic involution $\iota$ in its center,
\item 
$\Aut(X,C)/\left<\iota\right> \subset \Aut(\bP^1)$ acts via permutation on the $2g+2$ branch points of  the cover 
$C\rightarrow \bP^1$,
and
\item the induced action of $\Aut(X,C)$ on $J(C)[2]$ is induced by this permutation action.
\end{itemize}
\end{prop}
\begin{rema}
Suppose we diagonalize the forms
$$Q_1= \sum_{i=1}^{2g+2} x_i^2 , \quad Q_2 = \sum_{i=1}^{2g+2} \lambda_i x_i^2.$$
The combinations $Q_2-\lambda_i Q_1$ correspond to the $b_i$. The $2$-elementary extension
\begin{equation}
1 \rightarrow J(C)[2] \rightarrow H \rightarrow \left< \iota\right>\rightarrow 1 \label{eqn:2extension}
\end{equation}
acts on $\bP^{2g+1}$ via diagonal $(2g+2)\times(2g+2)$ matrices with $\pm 1$ as entries. The image in the quotient 
$\left<\iota \right>$ encodes the determinant of the matrix. 
\end{rema}

\section{Principal homogeneous spaces for the Jacobian} \label{sect:PHS}
We maintain the assumptions of Section~\ref{sect:geometry}.  
\subsection{Abstracting the principal homogeneous space}
\label{subsect:APHS}
Let $C\rightarrow \bP^1$ be a hyperelliptic curve of genus $g$; 
in the equivariant context, we assume that the $G$-action on $C$ descends to a linear
action on $\bP^1$.  The hyperelliptic involution $\iota$
induces an involution
\begin{equation} \label{eq:iota}
\begin{array}{rcl}
\iota: \Pic^1(C)& \rightarrow &\Pic^1(C) \\
               D & \mapsto & g^1_2  - D.
\end{array}
\end{equation}
Now suppose that $F$ is a square root of $\Pic^1(C)$, meaning a $J(C)$ principal homogeneous space satisfying
Wang's relation (\ref{eqn:Wang}). The automorphisms of $F$ include translations by $J(C)[2]$ and transformations
$$x \mapsto \iota(x)+\tau,  \quad  \tau \in \Pic^1(C), 2\tau = g^1_2,$$
encoded by the extension (\ref{eqn:2extension}).  

Given $C$, does there exist a smooth complete
intersection of two quadrics $X\subset \bP^{2g+1}$ whose associated pencil yields $C$?
A necessary condition is the existence of a $J(C)$-principal homogeneous space $F$ satisfying Wang's relation (\ref{eqn:Wang}). 
However, this is not the only way such a variety may arise. Consider
a Brauer-Severi variety $P \subset \bP^{\binom{2g+3}{2}-1}$, realized geometrically as a $2$-Veronese
reimbedding of $\bP^{2g+1}$, and a pencil of hyperplane sections of $P$. The base locus $X$ is geometrically 
a complete intersection of two quadrics in $\bP^{2g+1}$ and gives rise to auxiliary varieties $F_{g-1}(X)$ and $C$
as above. However $X$ is generally not embeddable in $\bP^{2g+1}$; in equivariant terms, a $G$-action
$$G\times X \rightarrow X$$
may not linearize to $\bP^{2g+1}$.  
This is often the only obstruction:
\begin{prop} \label{prop:reconstruct}
Let $C\rightarrow \bP^1$ be a hyperelliptic curve of genus $g$. 
In the equivariant context, we assume the group acts with a fixed point outside the branch locus.

A square root $F\rightarrow \Pic^1(C)$, as $J(C)$ principal homogeneous spaces, exists if and only if there exists 
a codimension-two linear section
$$
X \subset P\subset \bP^{\binom{2g+3}{2}-1},
$$
where $P$ is a form of $\bP^{2g+1}$ realized as a Veronese variety.
Thus there exists
$$X \subset \bP^{2g+1},$$
with the group acting linearly in $\bP^{2g+1}$ in the equivariant context, if and only if we may choose $P$ such that $[P]=0 \in \rH^2(\bG_m)$.  
\end{prop}
The cohomology group is $\rH^2(G,k^{\times})$ in the $G$-equivariant context and $\Br(k)$ over a nonclosed field $k$.  
\begin{proof}
We follow \cite{BGW}; the case of nonclosed fields is a corollary of their results.  
Given $p\in \bP^1$ ($k$-rational or $G$-fixed) that is not a branch point of $C\ra \bP^1$, write $p',p''\in C$ for the points over $p$.
We have an exact sequence for the generalized Jacobian of $C$ with respect to $\{p',p''\}$
\begin{equation} 0 \ra T \rightarrow J_m(C) \rightarrow J(C) \rightarrow 0 \label{eq:GJextension}
\end{equation}
where the first term 
$$T = (\mathbf{R}_{\{p',p''\}/\{p\}}\bG_m) / \bG_m.$$
Taking two-torsion gives
\begin{equation} 0 \ra \mu_2 \ra J_m(C)[2] \ra J(C)[2]\rightarrow 0. \label{gjacmod2}
\end{equation}
Suppose that $L$ is an \'etale algebra of degree $2g+2$ over $k$
associated with the branch points of $C\rightarrow \bP^1$.  We have \cite[Prop.~22]{BGW} identifications
\begin{align}
J_m(C)[2] &\Leftrightarrow (\mathbf{R}_{L/k} \mu_2)_{N=1} \\
J(C)[2] & \Leftrightarrow (\mathbf{R}_{L/k} \mu_2)_{N=1}/\mu_2
\end{align}
where $N:\mathbf{R}_{L/k} \ra \mu_2$ is the norm map from the restriction of scalars.  
These act linearly and projectively linearly on $\bP^{2g+1}$ respectively. 
The existence of $F$ is controlled by
$$0 \rightarrow J(C)[2] \rightarrow J(C)[4] \stackrel{\times 2}{\ra} J(C)[2] \rightarrow 0;$$
given $[\Pic^1(C)] \in \rH^1(J(C)[2])$, the obstruction to a square root $[F] \in \rH^1(J(C)[4])$ sits in
$$\rH^2(J(C)[2])=\rH^2((\mathbf{R}_{L/k} \mu_2)_{N=1}/\mu_2),$$
i.e., the Steenrod square of $[\Pic^1(C)]$. 
The vanishing of this class means $\bP^{2g+1}$ descends to a Brauer-Severi variety $P$.  
Moreover, the obstruction to producing $X\subset \bP^{2g+1}$ is controlled by \cite[Th.~24]{BGW}
$$0 \rightarrow J_m(C)[2] \rightarrow J_m(C)[4] \stackrel{\times 2}{\ra} J_m(C)[2] \rightarrow 0;$$
once $F$ exists, this is controlled via (\ref{gjacmod2}) by a class 
$$
\alpha \in \operatorname{coker}(\rH^1(J(C)[2])\rightarrow \rH^2(\mu_2)),
$$ 
where $\alpha\equiv [P]$ by the identifications.
\end{proof}

\begin{rema}What does this argument yield -- in the equivariant context -- when there is no fixed point?
Assume first that the cocycle in $\rH^2(G,J(C)[2])$ vanishes.  
Write $U\subset \bP^1$ for the complement of the branch points, with the induced $G$-action. The
generalized Jacobian $J_m(C)$ may still be defined over $U$ using (\ref{eq:GJextension}) with
$$T=(\mathbf{R}_{C\times_{\bP^1} U/U}\bG_m)/\bG_m.$$
We obtain 
$$
\xymatrix{
\cX  \ar@{^{(}->}[r] \ar[rd] & \cP \ar[d] \\
                                      & U
}
$$
where the vertical arrow is a Brauer-Severi fibration of relative dimension $2g+1$. If the cocycle in $\rH^2(G,J_m(C)[2])$ vanishes 
then the vertical arrow is a linear $\bP^{2g+1}$ fibration.
\end{rema}

\subsection{Projective geometry}
For the moment, we ignore the group action or assume the base field is algebraically closed. 
Recalling the imbedding $F_{g-1}(X) \subset \Gr(g,2g+2)$, we have
$$\cO_{\Gr(g,2g+2)}(1)|F_{g-1}(X) = \cO_{J(C)}(4\Theta),$$
where $\Theta$ is the class of a theta divisor.
Note however that the corresponding embedding is {\em not} linearly normal as
$$4^g = \dim \Gamma(\cO_{J(C)}(4\Theta)) > \dim \Gamma(\cO_{\Gr(g,2g+2)}(1)) = \binom{2g+2}{g}, \quad g>1.$$
For small $g$, we have 
$$
\begin{array}{r|cc}
g & 4^g & \binom{2g+2}{g} \\
\hline
1 & 4  & 4 \\
2 & 16 & 15 \\
3 & 64 & 56 \\
4 & 256 & 210 
\end{array}
$$

We explain the reason for this discrepancy. Suppose that $(J,\Theta)$ is a principally polarized abelian variety and 
$L$ is a line bundle on $J$ representing $\Theta$. For each $n\in \bN$, the Heisenberg extension associated with 
$n\Theta$
$$1 \rightarrow \bG_m \rightarrow \cG(L^n) \rightarrow J[n] \rightarrow 1$$
acts on the space of global sections $\Gamma(L^n)$. Recall that the extension data is given by the commutator
$$J[n] \times J[n] \rightarrow \mu_n \subset \bG_m$$
associated with the polarization form. Suppose that $n=4$ and realize $J[2] \subset J[4]$ in the standard way; the 
commutator pairing for $4\Theta$ is isotropic on $J[2]$, i.e., we may regard
$$\mu_2 \times J[2] \subset \cG(L^4)$$
as an abelian subgroup. Thus it is reasonable to diagonalize the theta functions for this group. Indeed, we have 
\begin{prop} \label{prop:BL} \cite[Ex.~6.10.1]{BLbook}
Let $\vartheta \in \Gamma(J,L)$ denote a generator and $\tau_x^*\vartheta$ its translate under $x\in J$. Then the
elements 
$$\{2^*\vartheta_x : x\in J[2] \}$$
form a basis for $\Gamma(J,L^4)$, naturally indexed by the $2$-torsion elements of $J$. (Here $2:J\ra J$ is multiplication by two.)
\end{prop}

\

Assume that either $k$ is nonclosed or that all the varieties and constructions are $G$-equivariant.
For our application, we use the squaring map
$$F_{g-1}(X) \rightarrow \Pic^1(C)$$
introduced in (\ref{morphism:Wang}); thus Proposition~\ref{prop:BL} applies. 

Consider the canonical theta divisor $\vartheta=\Sym^{g-1}(C) \subset \Pic^{g-1}(C)$. 
We analyze the translates of $\vartheta$ by elements in 
$$\left<b_1,\ldots,b_{2g+2}\right> \subset \Pic(C)$$
contained in $\Pic^1(C)$. 
These have the structure of a principal homogeneous space for $J(C)[2]$.
The Galois action on $\left<b_1,\ldots,b_{2g+2}\right>$
factors through the permutation representation on the branch points.  

Suppose first that $g$ is even; here the principal homogeneous space is trivial with distinguished divisor 
$$\vartheta - \frac{g-2}{2} g^1_2,$$
i.e., the canonical theta divisor translated by the $g^1_2$. Recall that $J(C)[2]$ corresponds to even partitions
$$S \sqcup S^c = \{b_1,\ldots,b_{2g+2} \}, \quad |S|=2j.$$
These are counted via the combinatorial identity for even $g$
$$4^g = \binom{2g+2}{g} + \sum_{j=0}^{g/2-1} \binom{2g+2}{2j}.$$
The sections of 
$$\Gamma(F_{g-1}(X),\cO_{F_{g-1}(X)}(1))$$
correspond to translates associated with sums of $g$ branch points $b_1,\ldots,b_{2g+2}$.  

\begin{rema}
For even $g$, any square root $F$ of $\Pic^1(C)$ admits a distinguished polarization of type $4\Theta$.  
However, given a projective representation $\rho:G \ra \PGL(V)$ note that 
$$\alpha(\wedge^g \rho) = g \alpha(\rho).$$
Thus two-torsion $\alpha(V) \in \rH^2(\bG_m)$ vanishes on passage from $X\subset \bP(V)$ to $F_{g-1}(X)\subset \bP(\wedge^g V)$.
\end{rema}

Now take $g$ odd.  The odd-degree divisors in $\left<b_1,\ldots,b_{2g+2}\right>$ -- a principal homogeneous space of $J(C)[2]$ -- correspond to 
odd partitions
$$S \sqcup S^c =\{b_1,\ldots,b_{2g+2} \}, \quad |S|=2j-1.$$
These index translates of $\vartheta$ pulling back to our desired polarization on $F_1(X)$. 
We have the combinatorial identity for odd $g$
$$4^g = \binom{2g+2}{g} + \sum_{j=1}^{(g-1)/2} \binom{2g+2}{2j-1}.$$
Again, the sections correspond to translates associated with sums of $g$ branch points.  

\begin{rema}For odd $g$, any square root $F$ of $\Pic^1(C)$ has a (Galois or $G$)
invariant divisor class: the pull back of the divisors 
$$\vartheta - \sum_{j\in J} b_j - \frac{g-2j-1}{2}g^1_2  \in \Pic^1(C), \quad |J|=2j-1,$$
to $F$. However, there may be an obstruction to the existence of a line bundle (defined over $k$
or linearized for $G$) on $X$ realizing this class.  Nonzero two-torsion $\alpha(V) \in \rH^2(\bG_m)$ remains nonzero on passage from $X\subset \bP(V)$ to $F_{g-1}(X)\subset \bP(\wedge^g V)$.
\end{rema}

 \section{Lagrangian interpretation in dimension three}
\label{sect:lagrangian}

In Section~\ref{subsect:APHS} we discussed how to recover a 
smooth complete intersection of two quadrics $X \subset \bP^{2g+1}$
from the associated hyperelliptic curve $C$, principal homogeneous space $J(C) \times F \ra F$, and additional
cohomological data. We present a geometric framework for these reconstruction
results when $g=2$.

The relationship between hyperk\"ahler manifolds $Y$ and Fano varieties arising as Lagrangian submanifolds
is rich and intricate. Lagrangian $\bP^n \subset Y$ can be characterized via intersection properties of the
Hodge lattice of $Y$ \cite{BHT,HTMMJ}. Further subtle constructions have been studied in \cite[\S 1.1]{FMOS}. 
For example, cubic fourfolds arise as Lagrangian submanifolds of hyperk\"ahler varieties of dimension eight
\cite{LLSvS}.   

We saw in Section~\ref{sect:geometry} that Kummer varieties arise naturally in the study of $X\subset \bP^{2g+1}$.
For $g=2$, Kummer {\em surfaces} take center stage 
but generalized Kummer {\em sixfolds} are most relevant for recovering $X$. We realize
$X$ naturally as a Lagrangian submanifold of a Kummer sixfold naturally arising from the variety
of lines $F_1(X)$. Recovering $X$ from $F=F_1(X)$ boils down to understanding certain Lagrangian subvarieties in
this Kummer sixfold.

\subsection{The basic construction}
Assume that the ground field is algebraically closed.

Let $X \subset \bP^5$ denote a smooth complete intersection of two quadrics, 
$F_1(X)$ its variety of lines, 
and $\Alb(F_1(X))$ the associated principally polarized abelian surface.
The variety of conics on $X$ equals the variety $F_2(q)$ in Section~\ref{sect:geometry}. Thus 
it fibers over a genus two curve $C$, parametrizing
connected components of the varieties of maximal isotropic subspaces in
the quadric hypersurfaces cutting out $X$. 
We may interpret $\Alb(F_1(X))\simeq J(C)$. 

Consider the Hilbert scheme $F_1(X)^{[4]}$ and the natural map
$$\begin{array}{rcl}
F_1(X)^{[4]} & \ra & J(C) \\
(\ell_1,\ell_2,\ell_3,\ell_4) & \mapsto & \ell_1+\ell_2+\ell_3+\ell_4 - h^2,
\end{array}
$$
where $h$ is the hyperplane class. Here $0$ is identified with cycles of lines obtained as
codimension-two linear sections of $X$.
The preimage of $0$ is a subvariety
$$\Kum(X) \subset F_1(X)^{[4]},$$
a twist of the generalized Kummer sixfold $K_{J(C)}(3)$ associated with $J(C)$.
Regarding $F_1(X)$ as a twist of $J(C)$ by a cocycle for 
$J(C)[4]$, applying this cocycle to $K_{J(C)}(3)$ yields $\Kum(X)$.

Consider the incidence variety
$$Z=\{[x,\ell]:x \in \ell \} \subset X \times F_1(X)$$
and the projection $\pi:Z \ra X$. This is generically finite
of degree four.
\begin{prop} The projection $\pi$ is flat over $X$.
\end{prop}
\begin{proof}
Indeed,
since $Z$ and $X$ are both smooth and projective it suffices
to show that $\pi$ is equidimensional.  But if a point
$x\in X$ were contained in a positive dimensional family of
lines then either 
\begin{itemize}
\item{$X$ admits a ruled hyperplane section through $x$ -- a cone
over an elliptic quartic curve -- which would force $X$ to
be singular at $x$;}
\item{$X$ admits a ruled surface of degree $<4$ through $x$,
violating the Lefschetz hyperplane theorem.}
\end{itemize}
Since $X$ is assumed to be smooth, we conclude the flatness of $\pi$.
\end{proof}

As a corollary, we obtain
\begin{prop}
There is an injective morphism
$$
\begin{array}{rcl}
j: X & \hookrightarrow & \Kum(X) \\
x & \mapsto & \pi^{-1}(x)
\end{array}
$$
realizing $X$ as a Lagrangian subvariety of $\Kum(X)$.
\end{prop}

\subsection{Numerical invariants}
Assume that $X$ is general in the sense that  $\NS(J(C))$ is of rank one, generated by $[\Theta]$.
Then the N\'eron-Severi group of $K_{J(C)}(3)$ has rank two and is
generated by \cite[p.~769]{BeauJDG}
\begin{itemize}
\item{$\theta$ -- subschemes with support along a theta
divisor;}
\item{$e$ -- where the nonreduced subschemes have
class $2e$.} 
\end{itemize}
\begin{prop}
The restriction homomorphism
$$j^*:\NS(\Kum(X)) \ra \NS(X) \simeq \bZ h$$
is given by
$$j^*(\theta)=5h, \quad j^*(e)=4h.$$
\end{prop}
\begin{proof}
For the purpose of this computation, we may ignore $G$-actions or work over an
algebraically closed field. 
The class $\theta$ on $J(C)=F_1(X)$ may be realized as the
locus $W_{\ell}\subset F_1(X)$ of lines incident
to a fixed line $\ell\subset X$, which sweeps out a divisor on $X$. 
This divisor is the exceptional divisor of the projection
$$\pi_{\ell}: X \stackrel{\sim}{\dashrightarrow} \bP^3;$$
the center of the inverse map is a quintic space curve $C$ of genus two
so the exceptional locus has degree $20$. This yields the
first equation.

The locus of nonreduced subschemes on $\Kum(X)$ restricts to the branch locus
of $\pi:Z\ra X$. Restricting to a line $\ell\subset X$, we see that
$$
\pi^{-1}(\ell) = \ell \sqcup W_{\ell},
$$
where the latter component has genus two and is realized as a degree-three
cover of $\ell$. Such a cover has eight branch points, so the branch locus
has class $8h$ and we get the second equation.
\end{proof}

\subsection{Reversing the construction}
For the moment, let $A$ be an abelian surface over $k$, not necessarily
principally polarized.
\begin{itemize}
\item{The group $A[4]$ acts on $A$ via translation.}
\item{
The semi-direct product $A[4] \rtimes \mu_2$,
where $\mu_2$ acts on $A[4]$ via $\pm 1$,
acts on $A$ as well.}
\item{Consider the addition map
$$\alpha:A^{[4]} \ra A$$
and the induced action of $G$ on $A^{[4]}$. 
Note that $A[4]$ acts on the fibers and the
action of $\mu_2$ commutes with addition.
Thus $G$ acts on $K_A(3)$ as well.}
\end{itemize}
Note that $G$ admits a distinguished normal $2$-elementary subgroup
$$
H = A[2] \times \mu_2.
$$
Now assume $A=J(C)$ and consider the various Lagrangian threefolds in $K_{J(C)}(3)$. 
The subgroup $H\simeq (\bZ/2\bZ)^5$ stabilizes each, acting via automorphisms.  
The full group $G$ gives an orbit of $16$ components, with transitive action of 
$$
G/H\simeq A[4]/A[2]\simeq A[2].
$$

\begin{rema}
Assume that the base field $k$ is arbitrary, of characteristic zero. 
Let $C$ be a smooth projective curve of genus two over $k$ and $F$ a principal homogeneous space
for $J(C)$ such that $2[F]=[\Pic^1(C)]$.  Let $K_{F}(3) \subset F^{[4]}$ denote the generalized Kummer variety lying
over the divisor $g^1_2 \in \Pic^2(C)$.  We may realize $F$ as the image of a cocycle for $J(C)[4]$;
the $16$ conjugate Lagrangian threefolds naturally form a principal homogeneous space 
for $J(C)[4]/J(C)[2]\simeq J(C)[2]$. However, even when this is trivial there is no guarantee that the corresponding
$X$ can be defined over $k$ -- the obstruction to descent is discussed in Proposition~\ref{prop:reconstruct}.  
\end{rema}

\section{Rationality constructions}
\label{sect:RatCon}
We continue to assume that $X\subset \bP^{2g+1}$ is a smooth complete intersection of two quadrics.

\subsection{Simple rational parametrizations}
\label{subsect:simple}
Fix a line $\ell \subset X$ and consider the projection 
$$\pi_{\ell}:X \stackrel{\sim}{\dashrightarrow} \bP^{2g-1},$$
a birational map, resolved by blowing up $\ell$.  
The inverse map 
$$\bP^{2g-1} \rightarrow \bP^{2g+1}$$
is obtained as follows: Consider a matrix
$$\left( \begin{matrix} l_{00} & l_{01} \\
				  l_{10} & l_{11} \\
				  q_0  & q_1 \end{matrix} \right),
				  $$
where the $l_{ij}$ are linear and and $q_i$ are quadratic.  The 
$2\times 2$ minors generate an ideal $I_Z$, where $Z$ is the base locus
of $\pi_{\ell}^{-1}$.  Cubic forms in $I_Z$ yield the linear series inducing this
map. The kernel of the matrix gives a rational map
$$\phi: Z \dashrightarrow \bP^1,$$
 which is regular for $g=2,3$. The generic fiber of $\phi$ is a quadric hypersurface
 in $\bP^{2g-3}$; we may interpret this as
 $\ell^{\perp}/\ell$, understood as a subquotient of the generic fiber of the pencil $q:\cQ \rightarrow \bP^1$
 (from Section~\ref{sect:geometry}), i.e.,
 the generic fibers of $\phi$ and $q$ are equivalent in the Witt ring of $k(\bP^1)$. 
 
 Much more can be said when $g=2$; we refer the reader to the corresponding case in Section~\ref{subsect:SRpara}.
 
 For the $g=3$ case, the fibration 
 $$\phi:Z \rightarrow \bP^1$$
 is a quadric surface fibration. We will denote by $C$ the discriminant curve
 of this fibration. Specifically, the relative variety of lines factors
 $$
 F_1(\phi) \stackrel{\varpi}{\rightarrow} C \rightarrow \bP^1,
 $$
 where $\varpi$ is a smooth conic fibration. 
 The following conditions are equivalent:
 \begin{itemize}
 \item{$\phi$ admits a section;}
 \item{$\varpi$ admits a section;}
 \item{$q:\cQ \ra \bP^1$ admits an isotropic plane containing $\ell \times \bP^1$.}
 \end{itemize}
 These are elementary properties of quadratic forms.
 The Amer-Brumer Theorem \cite[\S 2]{leep} gives a fourth equivalent condition:
 \begin{itemize}
 \item{there exists a plane $\bP^2 \subset X$.}
 \end{itemize}
 The fibration $\varpi$ admits a section if and only if we can express
 $$F_1(\phi) = \bP(\cE)$$
 for some rank-two vector bundle $\cE \ra C$. Desale-Ramanan \cite{DR} use such constructions to analyze 
 rank-two vector bundles of odd degree on hyperelliptic curves, relating automorphisms of $X$ to natural
 tensor and duality operations on the vector bundles. 
 
 \begin{rema}
Over nonclosed fields, e.g. $k=\bC(s)$, and for $g\ge 2$ (resp.~$g\ge 3$), 
it is possible for $X\subset \bP^{2g+1}$ (resp.~$F_1(X)\subset \Gr(2,2g+2)$) 
to admit a rational point even when $C$ admits no divisors of odd degree.

Consider a hyperelliptic curve $C_0 \rightarrow \bP^1$ represented as a double cover branched over
$g+1$ orbits for an involution on $\bP^1$.  For example, if the involution is 
$$[1,t] \mapsto [1,-t]$$
we could take the branch points as the roots $\pm \lambda_1, \ldots, \pm \lambda_{g+1}$ of
$\prod_{i=1}^{g+1} (t^2-a_i)$ with the $a_i$ distinct and nonzero.  Write
$$
X_0= \{\sum_{i=1}^{2g+2} x_i^2 = \sum_{i=1}^{g+1} \lambda_i (x_{2i-1}^2-x_{2i}^2) =0 \},
$$
with involution given by 
$$(x_1,x_2,x_3,x_4,\ldots,x_{2g+1},x_{2g+2}) \mapsto (x_2,x_1,x_4,x_3, \ldots x_{2g+2},x_{2g+1});$$
the associated hyperelliptic curve $C_0$ has the desired branch locus.

Choose a quadratic extension $L/k$ and let $X$ and $C$ denote the associated quadratic twists of 
$X_0$ and $C_0$.  By construction, $C$ admits no cycles of odd degree. However, $X$ is geometrically
rationally connected. The same holds for $F_1(X)$ provided $g\ge 3$
-- indeed, it has ample anticanonical class and thus is geometrically rationally connected. Thus both varieties
have $k$-rational points by the Graber-Harris-Starr Theorem.  

Note that \cite[Th.~29]{BGW} implies that $F_{g-1}(X)$ admits no $k$-rational points.  
 \end{rema}

% WHAT RESTRICTIONS ON THE GROUP? DOES IT MAKE SENSE TO LOOK AT $G$-EQUIVARIANT WITT GROUPS?}
% ADDRESS THIS IN A FUTURE PAPER

\subsection{Stable rationality constructions}

\label{subsect:SRpara}

Fix a $(g-1)$-dimensional subspace $L \subset X$ and look at the projection from $L$:
$$\pi: X \dashrightarrow \bP^{g+1}$$
with generic fiber a projective space $\bP^{g-2}$. 
We analyze the structure of this bundle.

Suppose that $L=\{x_0=\ldots=x_{g+1}\}$ so that the induced map on linear spaces factors through 
$$
\bP(\cO_{\bP^{g+1}}(-1) \oplus \cO_{\bP^{g+1}}^g) 
$$
so that $y_0,\ldots,y_{g-1}$ are trivializing sections of the $\cO_{\bP^{g+1}}$ factors and the linear series to $\bP^{2g+1}$ is given by
$$
y_0,\ldots,y_{g-1}, zx_0, \ldots, zx_{g+1}.
$$
The proper transform of $X$ has equations
$$
A(x_i;y_j) + z Q(x_j) = B(x_i;y_j) + z R(x_j)=0,
$$
where $A$ and $B$ are bilinear and $Q$ and $R$ are quadratic.  Eliminate $z$ to get the relation
$$
z = -A/Q = -B/R, 
$$
which gives 
$$
AR-BQ = F_0y_0+\cdots +F_{g-1}y_{g-1}=0,
$$
which is cubic in $x_i$ and linear in $y_j$; the $F_j$ are cubic forms in $x_0,\ldots,x_{g+1}$.  
This is the formula for the $\bP^{g-2}$ bundle in the product
$$
\bP^{g+1}_{x_i} \times \bP^{g-1}_{y_j}.
$$
Look at the locus $C$ where the morphism
$$
\Bl_L(X) \rightarrow \bP^{g+1}
$$
 fails to be flat. The $F_j$ are linear combinations of $Q$ and $R$ with linear coefficients. The locus $C$ is the residual intersection to the locus
$$
	Z = \{Q=R=0\}
$$
in the intersection of cubics 
$$
\{F_0=\cdots=F_{g-1}=0\}.
$$
\begin{prop}
The excess intersection contribution of $Z$ to the intersection of cubics is $3^g-(2g+1)$ whence $C$ has degree $2g+1$.  
The curve $C$ is hyperelliptic, embedded via a generic $(2g+1)$-degree polarization $D$.  
\end{prop}
\begin{proof}
This is a computation with Fulton's excess intersection formula, encoded by the exact sequence
$$
0 \rightarrow N_{Z/\bP^{g+1}}\simeq \cO_Z(2)^2 \rightarrow \cO_Z(3)^g \rightarrow Q \rightarrow 0,
$$
where the equivalence of $Z$ equals $c_{g-2}(Q)$. Note that
$$
(1+3ht)^g / (1+2ht)^2 = 1 + (3g-4)ht + \cdots + \frac{3^g - 2g -1}{4} h^{g-2}t^{g-2}$$
using the identity
$$\sum_{j+k=0}^{g-2} \binom{g}{k} 3^k  (j+1) (-2)^j = \frac{3^g - 2g -1}{4}.$$
\end{proof}

Fixing $(C,D)$, what are the constructions that arise? We analyze the induced morphism
$$\begin{array}{rcl}
\gamma: F_{(g-1)}(X) &\rightarrow& \Pic^{2g-1}(C) \\
    L & \mapsto & D.
    \end{array}
 $$ 
following \cite{Donagi} and \cite{BGW}.
First, translation by two-torsion in $J(C)[2]$ corresponds to an automorphism of $X$ acting trivially on cohomology. 
Thus we have
$$\gamma([L]+\tau) = \gamma([L]), \quad \tau \in J(C)[2].$$ 
Automorphisms with determinant $-1$ may be presented in the form
$$ p \mapsto -p+ \beta, \quad \beta \in \Pic^1(C), 2\beta = g^1_2.$$
These act by $-1$ on middle cohomology whence
$$\gamma(-[L]+\beta) = (2g+1)g^1_2 - D.$$
Moreover, Wang's formula $2[F_1(X)]=[\Pic^1(C)]$ \cite{wang} implies these are the only possible relations intertwining 
$\gamma$ for a generic curve $C$.  

To summarize our discussion:
\begin{prop}
For each hyperelliptic curve $C$ of genus $g$ and unordered pair of divisors 
$$D,D' \in \Pic^{2g+1}(C), \quad D+D' = (2g+1)g^1_2,$$
we obtain a group of stable birational equivalences of $\bP^{g+1}$ parametrized by the 
group 
$$(\bZ/2\bZ)^{2g+1} \simeq \left<b_1,\ldots,b_{2g+2}\right>/\left<g^1_2\right> \subset \Pic(C)/\left<g^1_2\right>.$$
\end{prop}  

\begin{rema}
When $g=2$, we obtain birational equivalences of $\bP^3$. This is the subgroup of the Cremona group on
$\bP^3$ associated with an orbit of $F_1(X)$ under the diagonalizable automorphisms of $X$.  If $\ell \subset X$
is a line and $h \in \Aut(X)$ is a diagonalizable automorphism then the associated birational map is
$$\bP^3 \stackrel{\pi_{\ell}}{\dashleftarrow} X \stackrel{h}{\rightarrow} X \stackrel{\pi_{\ell}}{\dashrightarrow} \bP^3.$$
\end{rema}

\section{Reduction to nonclosed fields}
\label{sect:NCfield}

We are interested in translating rationality criteria for geometrically rational threefolds over nonclosed fields to the equivariant context.
For example, the existence of points and subvarieties of prescribed type sometimes suffices to characterize rationality over nonclosed fields. We hope that the corresponding criteria for $G$-equivariant rationality are valid for varieties with $G$-action, where $G$ is a finite group.  We focus on situations where principal homogeneous spaces over abelian varieties control rationality.  

\subsection{Representations and monodromy groups} \label{subsect:FLLM}
We recall results of \cite{GLLM} on representations of monodromy groups for curves
with group action.

Let $\Sigma$ denote a smooth projective complex curve of genus $g\ge 2$ and fundamental group $T$. Consider a finite group $G$
and a surjective homomorphism $p:T\twoheadrightarrow G$ with kernel $R$. 
This is associated with a connected covering $\tilde{\Sigma}\rightarrow \Sigma$ with
homology \cite[Prop.~1.1]{GLLM}:
\begin{equation} \label{hurwitz}
\rH_1(\tilde{\Sigma},\bQ) \simeq \bQ \otimes_{\bZ} R/[R,R] \simeq \bQ^2 \oplus \bQ[G]^{2g-2}.
\end{equation}
This is an equivariant refinement of the Hurwitz formula due to Chevalley-Weil.  

We decompose $\bQ[G]$ using the theory of semisimple algebras \cite[\S 3.2]{GLLM}:
$$\bQ[G] \simeq \bQ \times \prod_{i=1}^{\ell} A_i,$$
where each $A_i\simeq \operatorname{Mat}_{m_i}(D_i)$, matrices over a division algebra.
Moreover, let $L_i$ denote the center of $D_i$, a number field. The index $i$ for the product encodes
types of nontrivial representations of $G$ that are irreducible over $\bQ$. 
The formula (\ref{hurwitz}) therefore yields
\begin{equation}
\label{decomp}
\rH_1(\tilde{\Sigma},\bQ) \simeq \bQ^{2g} \oplus \bigoplus_{i=1}^{\ell} A_i^{2g-2}.
\end{equation}

Each of the summands $A_i^{2g-2}$ comes with a natural skew-Hermitian structure with respect to
an explicit subfield $K_i \subset L_i$ of index $\le 2$ \cite[\S 3]{GLLM}.
For each index $i$, there is a distinguished algebraic group $\cG_{G,i}$, defined over $K_i$, 
parametrizing the automorphisms of $A_i^{2g-2}$ preserving this skew-Hermitian structure.
(We abuse notation, using the same notation for this group and its restriction of scalars to $\bQ$.)
The complex groups that may arise are listed in \cite[Thm.~1.7]{GLLM}:
\begin{equation}\label{listLie}
\operatorname{Sp}_{(2g-2)n}(\bC), \operatorname{O}_{(2g-2)n}(\bC), \operatorname{GL}_{(2g-2)n}(\bC)
\end{equation}
for some $n\in \bN$.  

Let $\mD_i \subset A_i$ denote the order arising as the image of $\bZ[G]$, $\cG_{G,i}(\mD_i)$
the resulting arithmetic group.  Note
that this arithmetic group depends only on the structure of $G$ and its action on the symplectic form.
Write $\Gamma_{G,p}$ for the mapping class 
group of $G$-coverings $\tilde{\Sigma} \ra \Sigma$ \cite[p.~1494]{GLLM}. 
For each $i=1,\ldots,\ell$, it admits a representation
$$\rho_{G,p,i}: \Gamma_{G,p} \ra \cG_{G,i} \subset \Aut_{A_i}(A_i^{2g+2}).
$$
We use $\cG^1_{G,i}$ to denote the elements of 
$\cG_{G,i}$ of reduced norm one over $L_i$; $\cG^1_{G,i}(\mD_i)$ has finite index in 
$\cG_{G,i}(\mD_i)$ \cite[Prop.~3.9]{GLLM} as the reduced norm takes values in roots of unity.

To summarize, we obtain natural representations of a finite-index subgroup of the mapping class 
group $\Gamma_{G,p}$ into the arithmetic groups $\cG_{G,i}(\mD_i)$ \cite[p.~1528]{GLLM}.
Fortunately, there are sufficient conditions guaranteeing that the image of $\rho_{G,p,i}$ contains a finite-index subgroup
of our arithmetic group.  
Assume further that $g\ge 3$ and $p$ factors
$$
p:T \stackrel{\varphi}{\rightarrow} F_g \stackrel{p'}{\rightarrow} G,
$$
where
\begin{itemize}
\item{$F_g$ is a free group on $g$ generators;}
\item{$\varphi$ is surjective;}
\item{the kernel of $p'$ contains one of the free generators of $F_g$.}
\end{itemize}
Whenever $G$ can be generated by $g-1$ elements we can find $p$
satisfying these conditions.  Under these assumptions, the image of $\rho_{G,p,i}$ contains a finite-index
subgroup of $\cG^1_{G,i}(\mD_i)$ \cite[Thm.~1.6]{GLLM}.

Observe that
$\cG^1_{G,i}(\mD_i) \subset \cG^1_{G,i}$ is Zariski dense by the Borel density theorem \cite[4.5.6, 5.1.11]{DWMbook}. 
Borel's Theorem requires that associated real Lie group has no compact factors; indeed, the assumption $g\ge 3$
guarantees the factors have $\bQ$-rank at least two \cite[p.~1529]{GLLM}. (Information about the real forms arising from these
groups may be found in \cite[\S~4]{GLLM}.)
 
The fact that the monodromy is large has implications for the structure of 
$\rH_1(\tilde{\Sigma},\bQ)$. The decomposition (\ref{decomp}) cannot be refined; any summand 
of $\rH_1(\tilde{\Sigma},\bQ)$ stable under the action of $\Gamma_{G,p}$ 
is a direct sum of the $\bQ^{2g}$ (coming from $\rH_1(\Sigma,\bQ)$) and the $A_i^{2g-2}$.
Using the classification (\ref{listLie}), we find that a very general covering 
$\tilde{\Sigma}\ra \Sigma$ associated with $p:T\ra G$,
the Jacobian $J(\tilde{\Sigma})$ admits no factor of dimension 
less than $g-1$.  

We summarize this, following \cite[Thm.~1.8]{GLLM}:
\begin{quote}
Fix a finite group $G$ and an integer $g\ge 3$ such that $G$ can be generated by
$g-1$ elements. There exists a family of pairs
$$(\Sigma, \tilde{\Sigma} \ra \Sigma),$$
where $\Sigma$ is a smooth complex projective curves of genus $g$ and $\tilde{\Sigma}\rightarrow \Sigma$
is a connected $G$-covering,
with the following property: For a very general $\tilde{\Sigma}$, the Jacobian $J(\tilde{\Sigma})$ admits no
factors over $\bQ$ of dimension less that $g-1$. 
\end{quote}

\subsection{Statement and proof of results}

Let $G$ be a finite group acting generically freely from the right on $P$,
an abelian variety. We do not assume that $G$ has a fixed point on $P$.
We write $\Alb(P)$ for the Albanese of $P$, the $G$-equivariant 
abelian variety parametrizing the group of translations on $P$.
In other words, $P$ is a $G$-equivariant principal homogeneous space
for $\Alb(P)$. 

Fix a base curve $B$, smooth and projective of genus $g\ge 3$ over $k$.
Let $f:\tilde{B} \ra B$ be a connected $G$ covering space, with $G$ acting
from the left. Consider the projection
$$
P \times \tilde{B} \rightarrow \tilde{B},
$$
where the left-hand-side has induced left $G$-action
$$
\gamma\cdot (p,\tilde{b}) = (p\gamma^{-1},\gamma \tilde{b}),
$$
with quotient $P\times_G \tilde{B}$.  
Consider the induced morphism
$$\pi_f: P\times_G \tilde{B} \ra B$$
whose fibers, away from the branch locus of $f$, are geometrically isomorphic
to $P$.  

\begin{prop} \label{prop:section}
Assume that $g > \dim(P)+1$, $g\ge 3$, and $B$ is of general moduli.

Suppose that for every $f$ as above, $\pi_f$ has a section.
Then the action of $G$ on $P$ has a fixed point.
\end{prop}
\begin{proof}
The existence of a section for $\pi_f$ is equivalent to a $G$-equivariant
morphism
$$\phi_f:\tilde{B} \rightarrow P$$
which factors
$$\tilde{B} \hookrightarrow \Pic^1(\tilde{B}) \ra P.$$

Our assumption -- that $\dim(P) < g-1$ -- allows us to apply the
results of Section~\ref{subsect:FLLM} to deduce $J(\tilde{B})$ has
no factors of dimension $\dim(P)$.  It follows that there is no nontrivial
$G$-equivariant homomorphism
$$J(\tilde{B}) \rightarrow \Alb(P).$$
This forces $\phi_f$ to be constant, which forces the triviality of $\Pic^1(\tilde{B})$.
\end{proof}

\section{Applications in dimension three}
\label{sect:ADT}
In this section, we present parametrizations arising from the existence of $G$-invariant linear subspaces
$$L \subset X \subset \bP^{2g+1}.$$
The geometry here is both simpler and richer than the geometry over nonclosed fields. There are more possible Galois actions
than actions via automorphisms; not every subgroup of $\fS_{2g+2}$ arises as the automorphisms of a configuration of 
$2g+2$ points. On the other hand, if one can find a linear subspace
$$L \simeq \bP^r \subset X$$
defined over $k$, one automatically has subspaces $\bP^s \subset X$ for all $s\le r$. This is not the case in the equivariant
context, as the underlying representation may be irreducible.  

Throughout this section, $k$ is algebraically closed of characteristic zero.

\subsection{Review of surface case} \label{subsect:RSC}
We review the classification of $G$-actions on smooth intersections of two quadrics in dimension $2$. 
The general strategy for attacking this question uses the $G$-equivariant minimal model program; the most systematic
description may be found in \cite{DI}.  

Let $X \subset \bP^4$ be a smooth quartic del Pezzo surface with a generically free action of a finite group $G$. 
There are $16$ lines on $X$, which are permuted by $G$.  If there exists a $G$-equivariant collection of 
disjoint lines, it may be blown down to obtain a del Pezzo surface of larger degree. Del Pezzo surfaces of
degree $7,8,$ or $9$ are equivariantly birational to $\bP^2$ or $\bP^1 \times \bP^1$. In degrees $5$ and $6$ there
are actions not birational to actions on homogeneous spaces as above; see Section 8 of \cite{DI} for more details.

For our purposes -- to illustrate the aspects common to all dimensions -- we focus on examples with generic automorphism 
group. Any quartic del Pezzo surface may be written in diagonal form
$$x_0^2+x_1^2+x_2^2+x_3^2+x_4^2 = a_0 x_0^2 + a_1 x_1^2 + a_2 x_2^2 + a_3x_3^2 + a_4x_4^2=0$$
which admits a diagonal action of $H=\mu_2^5/\mu_2$.  Generically, these are the only automorphisms.  
Hence we focus on subgroups of
$$
H\simeq (\bZ/2\bZ)^4
$$
acting via sign changes on coordinates of diagonal quadrics defining $X$. Consider involutions $\iota \in H$; we
present one representative for each conjugacy class:

\subsection*{(1,1,1,1,-1):} 
Here $X$ is a double cover of the quadric surface in $\bP^3$
$$(a_4-a_0)x_0^2 + (a_4-a_1)x_1^2 + (a_4-a_2)x_2^2 + (a_4-a_3)x_3^2$$
branched over the elliptic curve $E$ cut out by
$$x_0^2+x_1^2+x_2^2+x_3^2=0.$$
This surface has no fixed lines and has nontrivial cohomology $\rH^1(G,\Pic(X))$ 
(from the curve $E$, cf. \cite{BoPro}) and thus is not $G$-rational or even stably rational. It is clearly minimal as the orbits of lines consist of eight pairs meeting at points. (These are rulings of the quadric surface tangent to $E$.) 
Any $G$-action on a quartic del Pezzo surface containing this involution is $G$-irrational.

\subsection*{(1,1,1,-1,-1):} Here we have four fixed points given by $\{x_3=x_4=0\}$ and
eight orbits of disjoint lines. Thus $X$ is birational to a sextic del Pezzo surface, admitting three conic bundle structures.
One of these must be fixed under the involution,
thus $X$ is equivariantly birational to $\bP^1 \times \bP^1$ with linear action.

\subsection{Reduction to the case of nonclosed fields}

We turn to the case of dimension three. As a corollary of results in Section~\ref{sect:NCfield}, we have:

\begin{theo}
\label{thm:main-qua}
Let $X \subset \bP^5$ be a 
smooth complete intersection of two quadrics with  generically free action of a finite group $G$. 
Then $X$ is $G$-equivariantly birational to $\bP^3$, with $G$ acting projectively linearly,
if and only if there is a $G$-invariant line on $X$.
\end{theo}

\begin{proof}  
The action of $G$ on $X$ admits a canonical linearization on $\omega_X^{-1}=\cO_X(2)$. 
It follows that $G$ acts projectively on $\Gamma(\cO_X(1))$ and thus on $\bP^5$.  A central extension
$$1 \rightarrow \mu_2 \rightarrow \widetilde{G} \rightarrow G \rightarrow 1$$
acts linearly on $\bP^5$.  

An invariant line corresponds to a two-dimensional $\widetilde{G}$-invariant subspace. The complementary $G$-stable
subspace induces a projection
$$
\pi_{\ell}: X \stackrel{\sim}{\dashrightarrow} \bP^3,
$$ 
where $\bP^3$ admits a linear action of $\widetilde{G}$ and a projective action of $G$. 

We turn to the converse statement: Suppose there is a $G$-equivariant birational map
$$X \stackrel{\sim}{\dashrightarrow} \bP^3$$
with $G$ acting projectively linearly.

Let $B$ be a curve of large genus satisfying the conditions of Proposition~\ref{prop:section}. 
Construct an isotrivial family $\cX \rightarrow B$, with generic fiber isomorphic to $X$, split over a $G$-covering $\tilde{B} \ra B$.
(The $G$-action on $X$ induces one on $X\times \tilde{B}$; we take quotients to obtain our isotrivial family.)  
It is birational over $B$ to a fibration $\cP \rightarrow B$ that is generically a $\bP^3$-bundle. The Tsen-Lang
Theorem implies it is birational over $B$ to $\bP^3 \times B$.  
The variety of lines $F_1(\cX/B)$ is a principal homogeneous space over an abelian surface, over a dense open subset 
of $B$.  The main result of \cite{HT1} implies that $F_1(\cX/B)\rightarrow B$ admits a rational section. 
Proposition~\ref{prop:section} implies that the associated action of $G$ on $F_1(X)$ necessarily admits a fixed point,
which yields a line $\ell \subset X$ invariant under the $G$-action.
\end{proof}
\begin{rema}
Let $X$ have a $G$-action as above; we do not assume the existence of a fixed point or invariant line.
Nevertheless, we can construct the isotrivial fibration
$$\cX \ra B$$
which {\em always} admits a section by the Tsen-Lang (or Graber-Harris-Starr) Theorem.  
Thus Proposition~\ref{prop:section}, stated for abelian varieties, is not valid for other classes of varieties such as rationally-connected varieties.
\end{rema}

\subsection{Relations with the Burnside formalism}

The Burnside group formalism is presented in \cite{BnG}. Let $G$ be a finite group
acting generically freely on a smooth projective variety $Y$. For each nontrivial
subgroup $H \subset G$, we record data
\begin{itemize}
\item the locus $Z \subset Y$ fixed under $H$;
\item the action of $H$ on the normal bundle of each component of $Z$.
\end{itemize}
The character for the action on the normal bundle is called a {\em symbol}.  
Consider how this data changes as
we blow up $Y$ along a $G$-stable smooth subvariety. To obtain a 
$G$-birational invariant, we impose equivalences on the possible symbols that may arise,
the {\em blowup relations}.  

The simplest approach is to discard information about the components of the 
fixed locus $Z$, recording only the representation on the normal bundle.
(Here we only retain the dimensions of the components.). 
See \cite[Section 5]{HKTsmall} for refinements via Grothendieck classes
of varieties, and Section 6 of that paper for applications to cubic fourfolds.

\begin{prop} \label{prop:Burnside}
Let $G$ be a finite group acting generically freely on $X \subset \bP^5$,
a smooth complete intersection of two quadrics. Suppose there exist
\begin{itemize}
\item{an element $g\neq 1$ fixing a hyperplane section
$S\subset X$;}
\item{a subgroup $\left<g\right> \subset H \subset G$ acting on $S$.}
\end{itemize}
Assume that $S$ is not $H$-birational to either
\begin{enumerate}
\item{$\bP(V)$, a projectively linear representation of $H$; or}
\item{$\bP(E)$ where $E$ is an $H$-equivariant rank-two vector bundle over a curve.}
\end{enumerate}
Then $X$ is not $G$-birational to $\bP^3$ with a $G$-action.
\end{prop}
The surface case is addressed in \cite[Prop.~3.9]{KTprojective}.
\begin{proof}
Suppose we have a $G$-birational map
$$\rho: \bP^3 \stackrel{\sim}{\dashrightarrow} X$$
with $G$ acting projectively linearly on the source. By weak factorization, we may assume it is a composition of blowups and blowdowns along smooth $G$-stable centers.  

Consider the center $Z \rightarrow \bP^3$ blown up to obtain $S$; it is irreducible, fixed by $g$, and has an action of $H$.  Since $S \dashrightarrow Z$ is birational, we conclude that $S$ is birational to a divisor in $\bP^3$ with nontrivial stabilizer.
The locus in $\bP^3$ with nontrivial stabilizers is a union of linear subspaces, and we are in the first case.
If $Z$ is a point then $S$ is birational to the projectivization of the tangent space at that point; again, we are in the first case.
Finally, suppose $Z$ is a curve; then $S$ is birational to the projectivization of the normal bundle to $Z$, putting us in the second case.
\end{proof}

We turn to other representative examples. We follow Avilov \cite{avilov}, who enumerated actions of finite groups on three-dimensional
complete intersections of two quadrics known to be equivariantly
birational to projective space, a quadric hypersurface, or a Mori fiber space.

\begin{exam}
Suppose that $G$ admits a subgroup $H\simeq C_2 \times C_2$ acting on $X$ via the diagonal matrices
$$\operatorname{diag}(1,1,1,1,\pm 1, \pm 1).$$
Take 
$$g=\operatorname{diag}(1,1,1,1,1,-1)$$
which fixes a del Pezzo surface $S$. The residual $C_2$ action on $S$ was considered in Section~\ref{subsect:RSC}.
Since there is nontrivial cohomology, i.e. $\rH^1(C_2,\Pic(S))\neq 0$, the two possibilities from Proposition~\ref{prop:Burnside} are precluded.
It follows that $X$ is not equivariantly rational for any group containing $H$.
\end{exam}

We emphasize that Theorem~\ref{thm:main-qua} gives a {\em stronger} conclusion: 
If $X$ is $G$-birational to $\bP^3$ then $G$ acts on $F_1(X)$ with fixed points.
In particular, every element of $G$ fixes a point of $F_1(X)$ so $G$ has no elements conjugate to
$\pm \operatorname{diag}(1,1,1,1,-1,-1)$.
Indeed, these correspond to translates by two-torsion, which act freely.  
For instance, if $G=\left<(1,1,1,1,-1,-1)\right>$ the fixed locus on $X$
is an elliptic curve. There are no Burnside
invariants available as the relevant symbol groups for $G=C_2$ are zero \cite[Section 3.1]{HKTsmall}, \cite[Section 12]{KPT}.  

\begin{rema}
It would be interesting to have a general theory -- in the context of $G$-equivariant birational geometry --
encompassing both Burnside/symbol-type invariants and obstructions arising from Chow-theoretic 
principal homogeneous spaces for intermediate Jacobians.   
\end{rema}

\subsection{Rational complete intersections need not have points}
Theorem~\ref{thm:main-qua} shows that equivariant rationality of $X\subset \bP^5$ is governed by
the existence of invariant lines on $X$.  
When $G$ is cyclic, the existence of a $G$-invariant line guarantees a point on that line fixed by $G$.
For noncyclic groups, an invariant line need {\em not} have a fixed point.
There are examples of $G$-rational $X\subset \bP^5$ with no fixed points.

Let $C$ denote the complex curve
$$y^2=x^6+1.$$
Its automorphism group contains $G$,  a central extension 
$$1 \rightarrow \mu_2=\left<\iota\right> \rightarrow G \rightarrow D_{12} \rightarrow 1$$
of the dihedral group $D_{12}$ of order twelve. 
We write
$$G=\left<\sigma,\tau, \iota: \sigma^6=\tau^2=\iota^2=1, \tau \sigma \tau^{-1} \sigma = \iota \right>$$
with action
$$
\sigma\cdot (x,y) = (\zeta x, y), \quad \tau \cdot (x,y) = (y/x^3,1/x), \quad \zeta=e^{2\pi i /6},
$$  
where $\iota$ is the hyperelliptic involution.  The induced action on the global sections $\Gamma(\omega_C)$ is
$$\sigma \mapsto \left( \begin{matrix} \zeta & 0 \\
					0 & \zeta^2 \end{matrix} \right), \quad
	\tau \mapsto \left( \begin{matrix} 0 & -1 \\
						         -1 & 0 \end{matrix} \right), \quad
		\iota \mapsto \left( \begin{matrix} -1 & 0 \\
								  0 & -1 \end{matrix} \right).$$	
Actually, $\Aut(C)=G$; see
\cite[\href{https://www.lmfdb.org/Genus2Curve/Q/2916/b/11664/1}{Genus two curve 2916.b.11664.1}]{lmfdb} for more information.

Consider the quotient of $C$ under the unique cyclic subgroup of order three $\left<\sigma^2\right>$,
with invariants and equation
$$y,z=x^3 \quad y^2=z^2+1.$$	
Let $\rho:C \rightarrow R$ denote the corresponding degree-three morphism. The induced action on $R$
is generically free via the dihedral group of order eight. 

The double cover and $\rho$ give a morphism 
$$C \hookrightarrow \bP^1 \times R.$$
While $R$ is isomorphic to $\bP^1$ as a variety, the action of $G$ on $R$ is not linear. (The subgroup
$\left<\tau,\sigma\right>$ acts on $R$ as a Klein four-group, with no fixed points; such actions are not linearizable.)  
However, we do have a central extension 
$$1 \rightarrow \mu_2 \rightarrow \widetilde{G} \rightarrow G \rightarrow 1$$
and $\widetilde{G}$-representations $V=\Gamma(\omega_C)$ and $W$ such that
$$
C \hookrightarrow \bP(V) \times \bP(W) \hookrightarrow \bP(V\otimes W)\simeq \bP^3,
$$
realizing our curve as a $(2,3)$ divisor.  The linear series of cubic forms vanishing on $C$ gives a birational map
$$
\bP^3 \stackrel{\sim}{\dashrightarrow} X \subset \bP^5,
$$
where $X$ is a smooth complete intersection of two quadrics. This blows up $C$ and blows down 
$\bP(V)\times \bP(W)$ via projection to the first factor. The map $X\stackrel{\sim}{\dashrightarrow} \bP^3$
is projection from a $G$-stable line $\ell \simeq \bP(V)$.  

We claim that $X$ has no $G$-fixed points. We know that $\bP(V)$ has no fixed points so it suffices
to check that $\bP^3=\bP(V\otimes W)$ has no fixed points. Looking at $\sigma$ acting on $V$, 
any fixed points would necessarily lie on 
$$[1,0]\times W \text{ or }[0,1]\times W.$$
However, we have already seen that the dihedral group of order eight acts on $R$ without fixed points.  

This analysis also shows that the $G$-action on $X$ does not linearize to the ambient $\bP^5$.
The map $\Br_G(\mathrm{point}) \ra \Br_G(X)$ is not an isomorphism; its kernel contains
$0\neq \alpha(G,W) \in \rH^2(G,\mu_2)$ (see Proposition~\ref{prop:projectiveBrauer}).  This illustrates the general obstruction analysis 
in Section~\ref{subsect:APHS}. 

This answers a question of Avilov \cite[Rem.~2]{avilov}, who asked where $X$ -- Case 2(ii) of his Theorem 1-- fits in the equivariant 
birational classification.

\subsection{Another special example}
We return to another example highlighted by Avilov -- Case 2(iv) of \cite[Th.~1]{avilov} -- where $G$ fits into an exact sequence
$$1 \ra C_2^5 \ra G \ra \fS_4 \ra 1.$$
The associated binary sextic form is
$$T_0T_1(T_0^4-T_1^4).$$
The hyperelliptic curve 
$$C = \{U^2 = T_0T_1(T_0^4-T_1^4) \}$$
has automorphism group
$$\Aut(C) = \left< \sigma, \tau \right>,
$$
where 
$$\sigma(T_0,T_1,U) = (e^{3\pi i /4}T_0, e^{\pi i /4}T_1,U) \quad
\tau(T_0,T_1,U) = (\frac{1}{\sqrt{2}}(T_1-T_0), \frac{1}{\sqrt{2}}(T_0+T_1),U).$$
The resulting group has relations
$$\sigma^4=\iota, \tau^2=\iota^2=1, \iota \tau = \tau \iota, (\sigma \tau)^3=\iota,$$
where $\iota$ is the hyperelliptic involution. It sits in a central extension
$$1 \rightarrow \mu_2=\left<\iota\right> \rightarrow \Aut(C) \rightarrow \fS_4 \rightarrow 1.$$
This curve appears as
\cite[\href{https://www.lmfdb.org/Genus2Curve/Q/4096/b/65536/1}{Genus two curve 4096.b.65536.1}]{lmfdb}.

We analyze the induced action on the branch points, using coordinate $T_0/T_1$:
$$b_1=\{ 0 \}, b_2= \{\infty\}, b_3 = \{1 \}, b_4=\{i\}, b_5 = \{-1\}, b_6=\{-i\}.$$
The generators act via permutations
$$\sigma \mapsto (3456), \quad \tau \mapsto (13)(25)(46)$$
and the hyperelliptic involution acts trivially.  
Consider the induced action on
$$\Pic^1(C) \supset C$$
as in (\ref{eq:iota}).  
Note that $\iota$ fixes only the solutions to $L^2=g^1_2$, the $16$ points (cf.~Section~\ref{sect:geometry})
$$b_1,\ldots,b_6,b_1+b_2+b_3-g^1_2=b_4+b_5+b_6-g^1_2,\ldots,b_1+b_5+b_6-g^1_2=b_2+b_3+b_4-g^1_2.$$
None of these is simultaneously fixed by $\sigma$ and $\tau$, so $\Pic^1(C)$ admits no fixed point for $\Aut(C)$.

As for the complete intersection, we may take equations 
$$
X=\{Q_1=Q_2=0\} \subset \bP^5,
$$ 
where
$$Q_1 = x_0^2 + x_1^2 + i x_2^2 - x_3^2 - i x_4^2, \quad
Q_2 = x_1^2 + x_2^2 + x_3^2 + x_4^2 + x_5^2.$$
\begin{prop}
Suppose that $G$ acts on $X$ so that the induced homomorphism
$$G \rightarrow \Aut(C)$$
is surjective. Then $X$ is not $G$-equivariantly birational to
$\bP^3$.  
\end{prop}
This partly answers a question in \cite[Rem.~2]{avilov}; Avilov asked whether these are equivariantly birational to $\bP^3$.
\begin{proof}
By Theorem~\ref{thm:main-qua}, if $X$ is birational to $\bP^3$ then $F_1(X)$ admits a fixed point. But then 
$\Pic^1(C)$ would admit one as well via the squaring map $F_1(X) \ra \Pic^1(C)$. This would contradict the computation above.

The argument works under the weaker hypothesis that the image contains the $2$-Sylow subgroup
$$\left<\sigma, \tau\sigma^2\tau \right> \subset \Aut(C)$$
as
$\tau \sigma^2 \tau \mapsto (12)(46).$
\end{proof}

We may have rationality when smaller groups act. We restrict attention to the automorphism of order eight acting on coordinates by
$$\gamma=\left( \begin{matrix}
\alpha & 0 & 0 & 0 &  0 & 0 \\
0      & 0 & 0 & 0 & -1 & 0 \\
0      & 1 & 0 & 0 &  0 & 0 \\
0      & 0 & 1 & 0 &  0 & 0 \\
0      & 0 & 0 & 1 &  0 & 0 \\
0      & 0 & 0 & 0 &  0 & 1 
\end{matrix} \right), \quad \alpha=e^{3\pi i/4}.$$
This is chosen so that $\gamma\cdot Q_2=Q_2$ and 
$$\gamma\cdot Q_1 = \alpha^2 x_0^2 + x_2^2 + i x_3^2 - x_4^2 - i x_1^2 = -i Q_1.$$
Thus we may interpret $\gamma$ as a lift of $\sigma^{-1}$.  

What are the fixed points? 
The action on the underlying space via the contragredient representation is:
$$\left( \begin{matrix}
\alpha^7 & 0 & 0 & 0 &  0 & 0 \\
0      & 0 & 0 & 0 & -1 & 0 \\
0      & 1 & 0 & 0 &  0 & 0 \\
0      & 0 & 1 & 0 &  0 & 0 \\
0      & 0 & 0 & 1 &  0 & 0 \\
0      & 0 & 0 & 0 &  0 & 1 
\end{matrix} \right),$$
with eigenvectors
\begin{align*}
1: \ &  [0,0,0,0,0,1] \not \in X \\
\alpha: \ & p_1:=[0,1,\alpha^7,\alpha^6,\alpha^5,0]\in X \\
\alpha^3: \ & [0,1,\alpha^5,\alpha^2, \alpha^7,0] \not \in X \\
\alpha^5: \ & p_2:=[0,1,\alpha^3,\alpha^6,\alpha,0] \in X \\
\alpha^7: \ & [1,0,0,0,0,0], [0,1,\alpha,\alpha^2,\alpha^3,0]\not \in X.
\end{align*}
The last eigenspace is isotropic for $\{Q_2=0\}$ and thus meets $X$ in
two points $p_3,p_4:=[\pm 2i, 1,\alpha,\alpha^2,\alpha^3,0].$
The span of $p_1$ and $p_2$ is also isotropic for $\{Q_2=0\}$. 

\begin{prop}
The lines $\ell(p_2,p_3)$ and $\ell(p_2,p_4)$ are invariant under the action.
\end{prop}
 
Note however that 
$$\gamma^4=\left(
\begin{matrix} 
-1 & 0 & 0 & 0 & 0 & 0 \\
0 & -1 & 0 & 0 &  0 & 0 \\
0 & 0 & -1 & 0 &  0 & 0 \\
0 & 0 & 0 & -1 &  0 & 0 \\
0 & 0 & 0 & 0 &  -1 & 0 \\
0 & 0 & 0 & 0 &  0 & 1 
\end{matrix}
\right)
$$
acts on the variety of lines on $X$ with $16$ fixed points, i.e., the lines on the
quartic del Pezzo surface $X \cap \{x_5=0\}$.

This example is instructive in that the group
actions are not associated with Weyl groups:
\begin{itemize}
\item{$W(D_6)$ is  realized as the signed permutation
matrices with an even number of $(-1)$ signs, 
a semidirect product
$$
W(D_6) \simeq \fS_6 \ltimes C_2^5,
$$
where $C_2^5$ is the diagonal matrices;}
\item{the semidirect product 
$$\fS_6 \ltimes C_2^6/\text{Diagonal},$$
interpreted as a Weyl group for the projective orthogonal group.}
\end{itemize}
These are not the same; the Weyl group has a nontrivial central element
$$(-1,\ldots,-1)$$
whereas the latter group has {\em no} nontrivial central elements.  

However, {\em neither} of these actions coincide with our
situation! Both lack an element $\gamma$ of order eight,
sitting over a four-cycle of $\fS_6$, with $\gamma^4 \in H'$
(the $2$-elementary group of diagonal automorphisms) having nontrivial determinant.  
The automorphism group $\Aut(X)$ is NOT a semidirect
product of $\Aut(D\subset P^1)$ and $H'$, which
admits NO element $\gamma$ mapping to
$$ 
\left(
\begin{matrix} 
1 & 0 \\
0 & -i
\end{matrix}\right)
$$
and whose fourth power is diagonal with entries $(-1,-1,-1,-1,-1,1)$.
Indeed, the candidate in the Weyl group
$$\tilde{\gamma} = \left(
\begin{matrix} 
-1 & 0 & 0 & 0 & 0 & 0 \\
0 & 0 & 0 & 0 & -1 & 0 \\
0 & 1 & 0 & 0 &  0 & 0 \\
0 & 0 & 1 & 0 &  0 & 0 \\
0 & 0 & 0 & 1 &  0 & 0 \\
0 & 0 & 0 & 0 &  0 & 1 
\end{matrix}
\right)$$
induces 
$$\tilde{\gamma}^4 = \left(
\begin{matrix} 
1& 0 & 0 & 0 & 0 & 0 \\
0 & -1& 0 & 0 & 0 & 0 \\
0 & 0 & -1& 0 & 0 & 0 \\
0 & 0 & 0 & -1& 0 & 0 \\
0 & 0 & 0 & 0 &-1 & 0 \\
0 & 0 & 0 & 0 & 0 & 1 
\end{matrix}
\right)
$$
which has the wrong diagonal entries.

The rationality of these specific examples may be deduced in terms of the group actions:
\begin{prop}
Let $\gamma$ be an automorphism of order eight on $X$ such that
$\gamma^4$ fixes a smooth hyperplane section $S\subset X$ and 
the induced action on $S$ is by an element acting on the degeneracy locus (of both $X$ and $S$) as a four-cycle. Then $S$ admits a line invariant under $\gamma$.
\end{prop}
\begin{proof}
The fixed locus of $H'$ corresponds to a del Pezzo $S$ of degree 4
with action associated with the upper $5\times 5$ matrix.
Its lines correspond to weights in the $D_5$ root system:
$$
-e_i - e_j - e_k - e_l - e_m, 
 e_i + e_j - e_k - e_l - e_m,
 e_i + e_j + e_k + e_l - e_m.
 $$
 With respect to the standard basis of the Picard group $\{L,E_1,E_2,E_3,E_4,E_5\}$
 projected into $\Pic(S)/\bZ K_S$ we have
 $$L-E_i \mapsto e_i,  \quad 2L-E_j-E_k-E_l-E_m \mapsto -e_i.$$
 
 The induced automorphism has order four. The only possible elements of $W(D_5)$ of order four fix a line.  
 For example, the signed permutations
 $$\left( \begin{matrix}
 1 & 0 & 0 & 0 & 0 \\
 0  & 0 & 0 & 0 & -1 \\
 0   & 1 & 0 & 0 & 0 \\
 0  & 0  & 1 & 0 & 0 \\
 0 & 0 & 0 & -1 & 0 
 \end{matrix}
 \right),
 \left( \begin{matrix}
 1 & 0 & 0 & 0 & 0 \\
 0  & 0 & 0 & 0 & 1 \\
 0   & 1 & 0 & 0 & 0 \\
 0  & 0  & 1 & 0 & 0 \\
 0. & 0 & 0 & 1 & 0 
 \end{matrix}\right),
 \left( \begin{matrix}
 1 & 0 & 0 & 0 & 0 \\
 0  & 0 & 0 & 0 & -1 \\
 0   & -1 & 0 & 0 & 0 \\
 0  & 0  & -1 & 0 & 0 \\
 0 & 0 & 0 & -1 & 0 
 \end{matrix}\right)
 $$
 are all conjugate in $W(D_5)$ so it suffices to consider the
 first. This acts on $\Pic(S)$ by
 \begin{align*} 
 L &\mapsto 2L-E_1-E_3 -E_4 \\
 E_1 & \mapsto L-E_3-E_4 \\
 E_2 & \mapsto L - E_1 - E_4 \\
 E_3 & \mapsto L - E_1 - E_3 \\
 E_4 & \mapsto E_2 \\
 E_5 & \mapsto  E_5
 \end{align*}
 which leaves the lines $E_5$ and $2L-E_1-E_2-E_3-E_4-E_5$
 invariant.
 \end{proof}

 The assumption that the induced permutation of the degeneracy
 locus is a four-cycle is essential. The element
 $$\gamma_1=\left( \begin{matrix}
 1 & 0 & 0 & 0 & 0 \\
 0  & 0 & 0 & -1 & 0 \\
 0   & 0 & 0 & 0 & -1 \\
 0  & 1 & 0 & 0 & 0 \\
 0 & 0 & 1 & 0 & 0 
 \end{matrix}\right)
 $$
 acts via
  \begin{align*} 
 L &\mapsto 2L-E_1-E_4 -E_5 \\
 E_1 & \mapsto L-E_4-E_5 \\
 E_2 & \mapsto L - E_1 - E_5 \\
 E_3 & \mapsto L - E_1 - E_4 \\
 E_4 & \mapsto E_3 \\
 E_5 & \mapsto  E_2
 \end{align*}
 with orbits
 \begin{align*}
 E_1 \mapsto L-E_4-E_5 \mapsto& 2L-E_1-E_2-E_3-E_4-E_5 \mapsto L-E_2-E_3 \\
 E_2 \mapsto L-E_1-E_5 \mapsto& L-E_1-E_2 \mapsto E_5 \\
 E_3 \mapsto L-E_1-E_4 \mapsto& L-E_1-E_3 \mapsto E_4 \\
 L-E_2-E_4  \mapsto L-E_3-E_4 \mapsto& L-E_3-E_5 \mapsto L-E_2-E_5
  \end{align*}
 However, $\gamma_1$ maps to a product of two transposition in $\fS_6$.

%teaser for degree 18 case

\bibliographystyle{alpha}
\bibliography{EGodd2Q}

\begin{thebibliography}{{LMF}21}

\bibitem[Avi16]{avilov}
A.~Avilov.
\newblock Automorphisms of threefolds that can be represented as an
  intersection of two quadrics.
\newblock {\em Mat. Sb.}, 207(3):3--18, 2016.

\bibitem[BCDP18]{BCDP}
J\'er\'emy Blanc, Ivan Cheltsov, Alexander Duncan, and Yuri Prokhorov.
\newblock Finite quasisimple groups acting on rationally connected threefolds,
  2018.
\newblock \texttt{arXiv:1809.09226}.

\bibitem[Bea83]{BeauJDG}
Arnaud Beauville.
\newblock Vari\'{e}t\'{e}s {K}\"{a}hleriennes dont la premi\`ere classe de
  {C}hern est nulle.
\newblock {\em J. Differential Geom.}, 18(4):755--782 (1984), 1983.

\bibitem[BGW17]{BGW}
Manjul Bhargava, Benedict~H. Gross, and Xiaoheng Wang.
\newblock A positive proportion of locally soluble hyperelliptic curves over
  {$\mathbb Q$} have no point over any odd degree extension.
\newblock {\em J. Amer. Math. Soc.}, 30(2):451--493, 2017.
\newblock With an appendix by Tim Dokchitser and Vladimir Dokchitser.

\bibitem[BHT15]{BHT}
Arend Bayer, Brendan Hassett, and Yuri Tschinkel.
\newblock Mori cones of holomorphic symplectic varieties of {K}3 type.
\newblock {\em Ann. Sci. \'{E}c. Norm. Sup\'{e}r. (4)}, 48(4):941--950, 2015.

\bibitem[BL04]{BLbook}
Christina Birkenhake and Herbert Lange.
\newblock {\em Complex abelian varieties}, volume 302 of {\em Grundlehren der
  Mathematischen Wissenschaften [Fundamental Principles of Mathematical
  Sciences]}.
\newblock Springer-Verlag, Berlin, second edition, 2004.

\bibitem[BP13]{BoPro}
Fedor Bogomolov and Yuri Prokhorov.
\newblock On stable conjugacy of finite subgroups of the plane {C}remona group,
  {I}.
\newblock {\em Cent. Eur. J. Math.}, 11(12):2099--2105, 2013.

\bibitem[Bri18]{BrionLinearization}
Michel Brion.
\newblock Linearization of algebraic group actions.
\newblock In {\em Handbook of group actions. {V}ol. {IV}}, volume~41 of {\em
  Adv. Lect. Math. (ALM)}, pages 291--340. Int. Press, Somerville, MA, 2018.

\bibitem[BW19]{BenWit2}
Olivier Benoist and Olivier Wittenberg.
\newblock Intermediate {J}acobians and rationality over arbitrary fields.
\newblock {\em Annales scientifiques de l'\'Ecole normale sup\'erieure}, to
  appear, 2019.

\bibitem[BW20]{BenWit1}
Olivier Benoist and Olivier Wittenberg.
\newblock The {C}lemens-{G}riffiths method over non-closed fields.
\newblock {\em Algebr. Geom.}, 7(6):696--721, 2020.

\bibitem[CGR06]{ChGR}
V.~Chernousov, P.~Gille, and Z.~Reichstein.
\newblock Resolving {$G$}-torsors by abelian base extensions.
\newblock {\em J. Algebra}, 296(2):561--581, 2006.

\bibitem[CTS21]{CTSko}
Jean-Louis Colliot-Th\'el\`ene and Alexei Skorobogatov.
\newblock {\em The {B}rauer–{G}rothendieck Group}, volume~71 of {\em
  Ergebnisse der Mathematik und ihrer Grenzgebiete. 3. Folge. A Series of
  Modern Surveys in Mathematics}.
\newblock Springer International Publishing, 2021.

\bibitem[DI09]{DI}
I.~V. Dolgachev and V.~A. Iskovskikh.
\newblock Finite subgroups of the plane {C}remona group.
\newblock In {\em Algebra, arithmetic, and geometry: in honor of {Y}u. {I}.
  {M}anin. {V}ol. {I}}, volume 269 of {\em Progr. Math.}, pages 443--548.
  Birkh\"{a}user Boston, Boston, MA, 2009.

\bibitem[Don80]{Donagi}
Ron Donagi.
\newblock Group law on the intersection of two quadrics.
\newblock {\em Ann. Scuola Norm. Sup. Pisa Cl. Sci. (4)}, 7(2):217--239, 1980.

\bibitem[DR77]{DR}
U.~V. Desale and S.~Ramanan.
\newblock Classification of vector bundles of rank {$2$} on hyperelliptic
  curves.
\newblock {\em Invent. Math.}, 38(2):161--185, 1976/77.

\bibitem[FMOS21]{FMOS}
Laure Flapan, Emanuele Macr\`i, Kieran~G. O'Grady, and Giulia Sacc\`a.
\newblock The geometry of antisymplectic involutions, {I}.
\newblock {\em Mathematische Zeitschrift}, to appear, 2021.

\bibitem[FW71]{FW1}
A.~Fr\"{o}hlich and C.~T.~C. Wall.
\newblock Equivariant {B}rauer groups in algebraic number theory.
\newblock In {\em Colloque de {T}h\'{e}orie des {N}ombres ({U}niv. de
  {B}ordeaux, {B}ordeaux, 1969)}, pages 91--96. Bull. Soc. Math. France,
  M\'{e}m No. 25. 1971.

\bibitem[FW00]{FW2}
A.~Fr\"{o}hlich and C.~T.~C. Wall.
\newblock Equivariant {B}rauer groups.
\newblock In {\em Quadratic forms and their applications ({D}ublin, 1999)},
  volume 272 of {\em Contemp. Math.}, pages 57--71. Amer. Math. Soc.,
  Providence, RI, 2000.

\bibitem[GLLM15]{GLLM}
Fritz Grunewald, Michael Larsen, Alexander Lubotzky, and Justin Malestein.
\newblock Arithmetic quotients of the mapping class group.
\newblock {\em Geom. Funct. Anal.}, 25(5):1493--1542, 2015.

\bibitem[HKT21]{HKTsmall}
Brendan Hassett, Andrew Kresch, and Yuri Tschinkel.
\newblock Symbols and equivariant birational geometry in small dimensions.
\newblock In {\em Rationality of Varieties}, volume 342 of {\em Progress in
  Mathematics}, pages 201--236. Birkh\"auser, Cham., 2021.

\bibitem[HT13]{HTMMJ}
Brendan Hassett and Yuri Tschinkel.
\newblock Hodge theory and {L}agrangian planes on generalized {K}ummer
  fourfolds.
\newblock {\em Mosc. Math. J.}, 13(1):33--56, 189, 2013.

\bibitem[HT21a]{HT1}
Brendan Hassett and Yuri Tschinkel.
\newblock Cycle class maps and birational invariants.
\newblock {\em Communications on Pure and Applied Mathematics},
  74(12):2675--2698, 2021.

\bibitem[HT21b]{HT19}
Brendan Hassett and Yuri Tschinkel.
\newblock Rationality of complete intersections of two quadrics over nonclosed
  fields.
\newblock {\em Enseign. Math.}, 67(1-2):1--44, 2021.
\newblock With an appendix by Jean-Louis Colliot-Th\'{e}l\`ene.

\bibitem[KKLV89]{KKLV}
Friedrich Knop, Hanspeter Kraft, Domingo Luna, and Thierry Vust.
\newblock Local properties of algebraic group actions.
\newblock In {\em Algebraische {T}ransformationsgruppen und
  {I}nvariantentheorie}, volume~13 of {\em DMV Sem.}, pages 63--75.
  Birkh\"{a}user, Basel, 1989.

\bibitem[KKV89]{KKV}
Friedrich Knop, Hanspeter Kraft, and Thierry Vust.
\newblock The {P}icard group of a {$G$}-variety.
\newblock In {\em Algebraische {T}ransformationsgruppen und
  {I}nvariantentheorie}, volume~13 of {\em DMV Sem.}, pages 77--87.
  Birkh\"{a}user, Basel, 1989.

\bibitem[KP21]{KP1}
Alexander Kuznetsov and Yuri Prokhorov.
\newblock Rationality of {F}ano threefolds over non-closed fields.
\newblock In {\em Rationality of Varieties}, volume 342 of {\em Progress in
  Mathematics}, pages 249--290. Birkh\"auser, Cham., 2021.

\bibitem[KPT21]{KPT}
Maxim Kontsevich, Vasily Pestun, and Yuri Tschinkel.
\newblock Equivariant birational geometry and modular symbols.
\newblock {\em Journal of the European Mathematical Society}, online first,
  2021.

\bibitem[KT20]{BnG}
Andrew Kresch and Yuri Tschinkel.
\newblock Equivariant birational types and {B}urnside volume.
\newblock {\em Annali della Scuola Normale Superiore di Pisa}, to appear, 2020.

\bibitem[KT21]{KTprojective}
Andrew Kresch and Yuri Tschinkel.
\newblock Equivariant {B}urnside groups and representation theory, 2021.
\newblock \texttt{arXiv:2108.00518}.

\bibitem[Lee]{leep}
D.~Leep.
\newblock The {A}mer--{B}rumer theorem over arbitrary fields.
\newblock \url{http://www.ms.uky.edu/~leep/Preprints.html}.

\bibitem[LLSvS17]{LLSvS}
Christian Lehn, Manfred Lehn, Christoph Sorger, and Duco van Straten.
\newblock Twisted cubics on cubic fourfolds.
\newblock {\em J. Reine Angew. Math.}, 731:87--128, 2017.

\bibitem[{LMF}21]{lmfdb}
The {LMFDB Collaboration}.
\newblock The {L}-functions and modular forms database.
\newblock \url{http://www.lmfdb.org}, 2021.
\newblock [Online; accessed 23 July 2021].

\bibitem[Mor15]{DWMbook}
Dave~Witte Morris.
\newblock {\em Introduction to arithmetic groups}.
\newblock Deductive Press, [place of publication not identified], 2015.

\bibitem[Rei72]{ReidThesis}
Miles Reid.
\newblock {\em The complete intersection of two or more quadrics}.
\newblock PhD thesis, Trinity College, Cambridge, 1972.
\newblock available at
  \url{http://homepages.warwick.ac.uk/~masda/3folds/qu.pdf}.

\bibitem[RY00]{RYKS}
Zinovy Reichstein and Boris Youssin.
\newblock Essential dimensions of algebraic groups and a resolution theorem for
  {$G$}-varieties.
\newblock {\em Canad. J. Math.}, 52(5):1018--1056, 2000.
\newblock With an appendix by J\'{a}nos Koll\'{a}r and Endre Szab\'{o}.

\bibitem[RY02]{ReichsteinYoussin}
Z.~Reichstein and B.~Youssin.
\newblock A birational invariant for algebraic group actions.
\newblock {\em Pacific J. Math.}, 204(1):223--246, 2002.

\bibitem[Wan18]{wang}
Xiaoheng Wang.
\newblock Maximal linear spaces contained in the base loci of pencils of
  quadrics.
\newblock {\em Algebr. Geom.}, 5(3):359--397, 2018.

\end{thebibliography}

\end{document}